\documentclass[oneside, 10pt]{amsart}
\usepackage{latexsym, bbm, enumerate, amssymb, amsmath}
\usepackage[full,dcucite]{harvard} 
\usepackage[dvips]{graphicx}
\usepackage{setspace}
\usepackage{paralist}
\usepackage{ifpdf}

\usepackage[show]{ed}

\usepackage{tikz}

\theoremstyle{plain}
\newtheorem{theorem}{Theorem}
\newtheorem{corollary}[theorem]{Corollary}
\newtheorem{lemma}[theorem]{Lemma}
\newtheorem{proposition}[theorem]{Proposition}
\theoremstyle{definition}

\newtheorem{definition}[theorem]{Definition}

\newtheorem{remark}[theorem]{Remark}

\newcommand{\II}{\mathcal{I}}
\newcommand{\HH}{\mathcal{H}}
\newcommand{\MM}{\mathcal{M}}

\newcommand{\R}{\mathbb{R}}
\newcommand{\N}{\mathbb{N}}
\newcommand{\E}{{\mathbb{E}}}
\renewcommand{\P}{\mathbb{P}}
\newcommand{\F}{{\mathcal{F}}}

\newcommand{\Var}{\mathrm{Var}}

\newcommand{\ve}{v^e}
\newcommand{\vu}{v^u}

\newcommand{\we}{w^e}
\newcommand{\wu}{w^u}

\newcommand{\he}{h^e}
\newcommand{\hu}{h^u}

\newcommand{\Xu}{X^u}
\newcommand{\Xe}{X^e}

\newcommand{\Lu}{L^u}
\newcommand{\Le}{L^e}

\newcommand{\Mu}{M^u}
\newcommand{\Me}{M^e}

\newcommand{\Au}{A^u}
\newcommand{\Bu}{B^u}

\newcommand{\Ae}{A^e}

\renewcommand{\S}{\mathcal{S}}

\newcommand{\1}{\mathbbm{1}}

\newcommand{\equalbydef}{\stackrel{\rm def}{=}}

\newcommand{\subsectionnewline}{\mbox{}\medskip}

\let\oldmarginpar\marginpar
\renewcommand\marginpar[1]{\-\oldmarginpar[\raggedleft\footnotesize #1]%
{\raggedright\footnotesize #1}}

\definecolor{lightyellow}{RGB}{255,255,102}

\begin{document}

\title[]%
{Optimal Online Selection of a Monotone Subsequence: A Central Limit Theorem}
\author[]
{Alessandro Arlotto, Vinh V. Nguyen, \\
and J. Michael Steele}

\thanks{A. Arlotto: The Fuqua School of Business, Duke University, 100 Fuqua Drive, Durham, NC, 27708.
Email address: \texttt{alessandro.arlotto@duke.edu}}

\thanks{V. V. Nguyen: The Fuqua School of Business, Duke University, 100 Fuqua Drive, Durham, NC, 27708.
Email address: \texttt{vinh.v.nguyen@duke.edu}}

\thanks{J. M.
Steele:  Department of Statistics, The Wharton School, University of Pennsylvania, 3730 Walnut Street, Philadelphia, PA, 19104.
Email address: \texttt{steele@wharton.upenn.edu}}

\citationmode{full}

\begin{abstract}
        Consider a sequence of $n$ independent random variables with a common continuous distribution $F$, and consider
        the task of choosing an increasing subsequence where the observations are revealed sequentially and
        where an observation must be accepted or rejected when it is first revealed.
        There is a unique selection policy $\pi_n^*$ that is optimal in the sense that it
        maximizes the expected value of $L_n(\pi_n^*)$, the number of selected observations.
        We investigate the distribution of $L_n(\pi_n^*)$;
        in particular, we obtain a central limit theorem for $L_n(\pi_n^*)$ and a detailed understanding of its
        mean and variance for large $n$. Our results and methods are complementary to the work of \citeasnoun{BruDel:SPA2004} where an
        analogous central limit theorem is found for monotone increasing selections
        from a finite sequence with cardinality $N$ where $N$ is a Poisson
        random variable that is independent of the sequence.

        \bigskip

        \noindent {\sc Key Words.} Bellman equation, online selection, Markov decision problem,
        dynamic programming, monotone subsequence, de-Poissonization,
        martingale central limit theorem, non-homogeneous Markov chain.

        \bigskip

        \noindent {\sc Mathematics Subject Classification (2010).}
        Primary: 60C05, 60G40, 90C40; Secondary:  60F05, 60G42, 90C27, 90C39

\end{abstract}

\date{
first version: August 28, 2014; this version: March 29, 2015.
}

\maketitle



\section{Introduction}

In the problem of \emph{online} selection of a \emph{monotone increasing} subsequence,
a decision maker observes a sequence of independent non-negative random variables $\{X_1, X_2, \ldots, X_n\}$
with common continuous distribution $F$, and the task is to select a subsequence $\{X_{\tau_1}, X_{\tau_2}, \ldots, X_{\tau_j}\}$
such that
$$
X_{\tau_1} \leq  X_{\tau_2} \leq \cdots \leq  X_{\tau_j}
$$
where the indices $1 \leq \tau_1 < \tau_2 < \cdots < \tau_j \leq n$ are stopping times
with respect to the $\sigma$-fields $\F_i = \sigma\{X_1, X_2, \ldots, X_i \}$, $1 \leq i \leq n$.
In other words, at time $i$ when the random variable $X_i$ is first observed, the decision maker
has to choose to accept $X_i$ as a member of the monotone increasing sequence that is under construction,
or to reject $X_i$ from any further consideration.

We call such a sequence of stopping times a \emph{feasible policy}, and we denote the set of all such policies
by $\Pi(n)$.
For any $\pi \in \Pi(n)$, we then let $L_n(\pi)$ be the random variable that counts the number of selections made by policy $\pi$ for the realization
$\{X_1,X_2, \ldots, X_n\}$; that is,
$$
L_n(\pi)=\max \{j: X_{\tau_1} \leq  X_{\tau_2}\leq \cdots \leq  X_{\tau_j} \text{ where } 1\leq \tau_1 < \tau_2 < \cdots < \tau_j \leq n \}.
$$
\citeasnoun{SamSte:AP1981} found that for each $n \geq 1$
there is a unique policy $\pi^*_n \in \Pi(n)$ such that
\begin{equation}\label{eq:uniquepolicy}
\E[L_n(\pi^*_n)] = \sup_{\pi \in \Pi(n)} \E[L_n(\pi)],
\end{equation}
and for such optimal policies one has
\begin{equation}\label{eq:SSasymp}
\E[L_n(\pi^*_n)] \sim ( 2n )^{1/2} \quad \text{as } n \rightarrow \infty.
\end{equation}
\citeasnoun{BruRob:AAP1991} and \citeasnoun{Gne:JAP1999} showed that one actually has the crisp upper bound
\begin{equation}\label{eq:GnedinUpperBound}
\E[L_n(\pi^*_n)] \leq (2n)^{1/2} \quad \quad \text{ for all } n \geq 1,
\end{equation}
and, as corollaries of related work, \citeasnoun{RheTal:JAP1991}, \citeasnoun{Gne:JAP1999} and \citeasnoun{ArlSte:CPC2011}
all found that there is an asymptotic error rate for the lower bound
\begin{equation}\label{eq:RheeTalagrandLowerBound}
(2n)^{1/2} - O(n^{1/4}) \leq \E[L_n(\pi^*_n)] \quad \quad  \text{ as } n \rightarrow \infty.
\end{equation}
Here, our main goal is to show that $L_n(\pi^*_n)$ satisfies a central limit theorem.

\begin{theorem}[Central Limit Theorem for Optimal Online Monotone Selections]\label{thm:OurCLT}
For any continuous distribution $F$ one has for $n \rightarrow \infty$ that
\begin{equation}\label{eq:mean-bounds-main-theorem}
(2n)^{1/2} - O(\log n) \leq \E[L_n(\pi^*_n)] \leq (2n)^{1/2},
\end{equation}
\begin{equation}\label{eq:variance-bounds-main-theorem}
\frac{1}{3} \E[L_n(\pi^*_n)]  - O(1) \leq \Var[L_n(\pi^*_n)] \leq \frac{1}{3} \E[L_n(\pi^*_n)]  + O( \log n),
\end{equation}
and one has the convergence in distribution
\begin{equation}\label{eq:TheCLT}
\frac{ 3^{1/2}  \{ L_n(\pi^*_n) - ( 2 n )^{1/2} \} }{ (2 n)^{1/4}} \Longrightarrow N(0,1).
\end{equation}
\end{theorem}

Two connections help to put this result in context.
First, it is useful to recall the analogous problem of
offline (or full information) subsequence selection, for which there is a remarkably rich literature. Second, there
are closely related results of \citename{BruDel:SPA2001} (\citeyear*{BruDel:SPA2001}, \citeyear*{BruDel:SPA2004}) that deal with sequential
selection where the number of values to be seen is random with a Poisson distribution.

\subsection*{\sc First Connection: The Tracy-Widom Law}\subsectionnewline

If one knows all of the values $\{X_1, X_2, \ldots, X_n\}$
at the time the selections begin, then decision maker can select a
maximal increasing subsequence with length
\begin{equation}\label{eq:LnOffLineDef}
L_n=\max \{j: X_{i_1} \leq  X_{i_2}\leq  \cdots \leq  X_{i_j} \text{ where } 1\leq i_1 < i_2 <  \cdots < i_j \leq n \}.
\end{equation}
This \emph{full information} or \emph{offline} length $L_n$ has been studied extensively.

The question of determining the distribution of  $L_n$ was first raised by \citeasnoun{Ulam1961}, but the analysis
of $L_n$ was taken up in earnest
by \citeasnoun{Ham:BS1972}, \citeasnoun{Kin:AP1973},
\citeasnoun{LogShe:AM1977}, and \citeasnoun{VerKer:DAN1977} who  established in steps that
\begin{equation*}
\E[L_n] \sim 2  n^{1/2} \quad \text{as } n \rightarrow \infty.
\end{equation*}
Much later H.~Kesten conjectured \citeaffixed[p.~416]{AldDia:AMS1999}{cf.} that there should be positive constants
$\alpha$ and $\beta$ such that
\begin{equation}\label{eq:KestenConjecture}
\E [L_n] = 2 n^{1/2} - \alpha n^{1/6} + o( n^{1/6}) \quad \text{and} \quad
\left\{\Var[L_n]\right\}^{1/2} = \beta n^{1/6} + o( n^{1/6}).
\end{equation}
After subtle progress by  \citeasnoun{Pil:JCTA1990},
\citeasnoun{BolBri:AAP1992}, \citeasnoun{Kim:JCTA1996}, \citeasnoun{BolJan:CGP1997}, and \citeasnoun{OdlRai:CM2000}
this conjecture was settled affirmatively by \citeasnoun{BaiDeiKur:JAMS1999} who proved moreover that
$n^{-1/6} (L_n -2 n^{1/2})$
converges in distribution to the famous Tracy-Widom law which had emerged just a bit earlier
from the theory of random matrices. The recent monograph of \citeasnoun{Rom:CUP2015}
gives a highly readable account
of this development.

One distinction between the online and the offline problems is that, while the means are of the same order in each case, the variances are not of the same order. The standard deviation for offline
selection is of order $n^{1/6}$, but by \eqref{eq:variance-bounds-main-theorem} the standard deviation for the online selection is of order $n^{1/4}$.
Intuitively this difference reflects greater
uncertainty in the online selection problem than in the offline problem, but it is harder
to  imagine why moving to the online formulation would drive one all of the way from the Tracy-Widom law to
the Gaussian law.

\subsection*{\sc Second Connection: The Bruss-Delbaen Central Limit Theorem}\subsectionnewline

Consider the problem of sequential selection of a monotone increasing subsequence from $\{X_1, X_2, \ldots, X_{N_\nu} \}$
where ${N_\nu}$ is a Poisson random variable with mean $\nu$ that is independent of the sequence $\{X_1, X_2, \ldots \}$.
Just as in \eqref{eq:uniquepolicy} there is a unique
sequential policy  that maximizes
the expected number of selections that are made.
If we denote this optimal policy by $\pi^*_{N_\nu}$ then as before $L_{N_\nu} (\pi^*_{N_\nu})$
is the number of selections from $\{X_1, X_2, \ldots, X_{N_\nu}\}$ that are made by the policy $\pi^*_{N_\nu}$.

\citeasnoun{BruDel:SPA2001} proved that, as $\nu  \rightarrow \infty$, one has the mean estimate
\begin{equation}\label{eq:BD-optimal-mean}
\E[L_{N_\nu} (\pi^*_{N_\nu})] = (2 \nu)^{1/2} + O(\log \nu),
\end{equation}
and the variance estimate
\begin{equation*}
\Var[L_{N_\nu} (\pi^*_{N_\nu})] = \frac{1}{3}(2 \nu)^{1/2} + O(\log \nu).
\end{equation*}
Moreover, \citeasnoun{BruDel:SPA2004} proved that, as $\nu  \rightarrow \infty$, one has
the convergence in distribution
\begin{equation*}
\frac{ 3^{1/2}  \{ L_{N_\nu} (\pi^*_{N_\nu}) - ( 2 \nu)^{1/2} \} }{ (2 \nu)^{1/4}} \Longrightarrow N(0,1).
\end{equation*}

One needs to ask if it is possible to
``de-Poissonize" these results to get Theorem \ref{thm:OurCLT}, either in whole or in part.
We show in Section \ref{se:Mean-DePoissonization} that the lower half of \eqref{eq:mean-bounds-main-theorem} can be obtained
from \eqref{eq:BD-optimal-mean} by an easy de-Poissonization argument; in fact,
this is the only proof we know of this bound. In Section \ref{se:Mean-DePoissonization} we also explain as best we can, why no
further parts of Theorem \ref{thm:OurCLT} can be obtained by de-Poissonization.

One can further ask if it might be possible to adapt the \emph{methods}
of \citename{BruDel:SPA2001} (\citeyear*{BruDel:SPA2001}, \citeyear*{BruDel:SPA2004}) to prove Theorem \ref{thm:OurCLT}. The
major benefit of a Poisson horizon is that it gives access to the tools of continuous time Markov processes such as
the infinitesimal generator and Dynkin's martingales. Moreover, in this instance
the associated value function $V(t,x)$ can be written as a function of one variable by the space-time transformation $V(t,x) =
\bar{V}(t(1-x))$.

Here we lack these benefits. We work in discrete time with a known finite horizon,
and our value function $v_k(s)$ permanently depends on the state $s$ and the time to the horizon $k$.
This puts one a long way from the world of
\citename{BruDel:SPA2001} (\citeyear*{BruDel:SPA2001}, \citeyear*{BruDel:SPA2004}). Still,
in Section \ref{se:inferences-uniform-model} we give a brief
proof of the
well-known upper bound \eqref{eq:GnedinUpperBound} that echoes an argument of \citeasnoun[pp.~291--292]{BruDel:SPA2004}.
This seems to be the only instance of an overlap of technique.

\subsection*{\sc Organization of the Analysis}\subsectionnewline

The proof of our central limit theorem has two phases.
In the first phase, we investigate the analytic properties of the value functions
given by framing the selection problem as a Markov decision problem.
Section \ref{se:DP-general-f} addresses the monotonicity and the submodularity of the value functions.
We also obtain that the map $n \mapsto \E[L_{n} (\pi^*_{n})]$ is concave, and
this is used in Section \ref{se:Mean-DePoissonization} to prove the lower half of
\eqref{eq:mean-bounds-main-theorem}; this is our only de-Poissonization argument.

Sections \ref{se:smmoothness-of-value-functions} and \ref{se:DP-special-properties} develop
smoothness and curvature properties of the value functions.
In particular, we find that in the uniform model the value functions are concave as a function of the state variable,
but, for the exponential model, they are convex. This broken symmetry
is surprisingly useful even though the distribution of $L_n(\pi^*_n)$ does not depend on the model distribution $F$.

The second phase of the proof deals with a natural martingale that one obtains from the value
functions. This martingale is defined in Section \ref{se:probabilistic-interpretation-and-Optimality-martingale}, and it is used in
Sections \ref{se:inferences-uniform-model}, \ref{se:inferences-exponential-model} and \ref{se:Variance-Bounds-in-General}
to estimate the conditional variances of $L_n(\pi_n^*)$.
These estimates and a martingale central limit theorem are then used in Section \ref{se:CLT} to complete the proof of Theorem \ref{thm:OurCLT}.
Finally, in Section \ref{se:Conclusions} we comment briefly on two
open problems and the general nature of the methods developed here.

\section{Structure of the Value Functions}\label{se:DP-general-f}


We now let $v_k(s)$ denote the expected value of the number
of monotone increasing selections under the optimal policy when (i) there are $k$ observations 
that remain to be seen and (ii) the value of the most recently selected observation is equal to $s$.
The functions $\{ v_k: 1 \leq k < \infty \}$ are called the \emph{value functions},
and they can be determined recursively.
Specifically, we have the \emph{terminal condition}
$$
v_0 ( s )  = 0 \quad \quad \text{ for all } s \geq 0,
$$
and if we set $F(s)=P(X_i \leq s)$ then for all  $k \geq 1$ and $s\geq 0$ we have the recursion
\begin{equation}\label{eq:Bellman-general-f}
v_k( s )   = F( s ) v_{ k - 1 }( s ) + \int_s^\infty \max \{  v_{ k - 1 }( s ), 1 + v_{ k - 1 }( x ) \} \, dF( x ).
\end{equation}
To see why this equation holds, note that with probability $F(s)$ one is presented
at time $i = n - k + 1$ with a value $X_i$ that is
less than the previously selected value $s$. In this situation, we do not have the opportunity to select $X_i$.
This leaves us with $k-1$ observations to be seen and with the value of the last selected observation, $s$, unchanged.
This possibility contributes the term $F( s ) v_{ k - 1 }( s )$ to our equation.

Now, if the newly presented value satisfies $s \leq X_i$ then we have the \emph{option}
to select or reject $X_i = x$.
If we select $X_i=x$, then the sum of our present reward and expected future reward is  $1 + v_{ k - 1 }( x )$.
On the other hand, if  we choose not to select $X_i = x$, then we have no present reward and the expected future reward is $ v_{ k - 1 }( s )$
since the value of the running maximum is not changed.
Since $X_i$ has distribution $F$,
the expected optimal contribution is given by the second term of equation \eqref{eq:Bellman-general-f}.

The identity \eqref{eq:Bellman-general-f} is called the \emph{Bellman recursion} for the sequential selection problem.
In principle, it tells us everything there
is to know about the value functions; in particular, it determines
$$
\E[L_n(\pi^*_n)] = v_n(0) \quad \quad \text{for all } n \geq 1.
$$
Qualitative information can also be extracted from the recursion \eqref{eq:Bellman-general-f}.
For example, it is immediate from \eqref{eq:Bellman-general-f} that the value functions are always continuous.
More refined properties of the value functions may depend on $F$, and here it is often useful
to consider a special subclass of distributions.

\begin{definition}[Admissible Distribution]
A distribution $F$ is said to be \emph{admissible} if
there is an open interval $\II \subseteq [0, \infty)$ such that
\begin{enumerate}[(i)]
    \item $F$ is differentiable on $\II$,
    \item $F'(x) = f(x) > 0$ for all $x \in \II$, and
    \item $\int_{\II} f(x) \ dx = 1$.
\end{enumerate}
\end{definition}

The next lemma illustrates how admissibility can be used. The result is largely intuitive, but the formal proof
via \eqref{eq:Bellman-general-f} suggests that some care is needed.

\begin{lemma}[Monotonicity of Value Functions]\label{lm:value-function-decreasing-general-f}
For any distribution $F$ the value functions are non-increasing.
Moreover, if $F$ is admissible, then the value functions are strictly decreasing on $\II$.
\end{lemma}

\begin{proof}
The first assertion is trivial, so we focus on the second.
To organize our induction we denote by $\HH_k$ the assertion
\begin{equation*}
v_k( s + \epsilon) < v_k (s) \quad \text{for all $s \in \II$ and all $\epsilon >0$.}
\end{equation*}
When $k=1$, we have $v_1(s) = 1 - F(s)$, and admissibility
of $F$ implies $v_1$ is strictly decreasing on $\II$.
This establishes the base case  $\HH_{1}$.

For $k>1$ we assume that $\HH_{k-1}$ holds, and we note by the
Bellman recursion \eqref{eq:Bellman-general-f} and the characterizing properties of admissible distributions that
\begin{align*}
v_k( s + \epsilon) - v_k (s)
 = {} &  F(s+\epsilon) v_{k-1}( s + \epsilon) + \! \int_{s+\epsilon}^\infty \!\!\!\!\! \max\{  v_{k-1}( s + \epsilon), 1 + v_{k-1}( x ) \} f(x) \, dx \\
\notag
 & - F(s) v_{k-1}( s ) - \int_{s}^\infty \max\{ v_{k-1}( s ), 1 + v_{k-1}( x ) \} f(x) \, dx\\
\notag
\leq {} & F(s+\epsilon) v_{k-1}( s + \epsilon)  +  \int_{s+\epsilon}^\infty \!\!\!\! \max\{  v_{k-1}( s ), 1 + v_{k-1}( x ) \} f(x) \, dx \\
\notag
 &  - F(s + \epsilon) v_{k-1}( s ) - \int_{s +\epsilon}^\infty \max\{ v_{k-1}( s ), 1 + v_{k-1}( x ) \} f(x)\, dx\\
\notag
 = {} & F(s+\epsilon)\left\{ v_{k-1}( s + \epsilon)  - v_{k-1}( s ) \right\},
\end{align*}
where we first used
$v_{k-1}(s + \epsilon) <v_{k-1}(s)$
and then used the trivial estimate
$$ \{F(s + \epsilon) - F(s) \}  v_{k-1}( s ) \leq \int_s^{s+\epsilon} \max\{ v_{k-1}( s ), 1 + v_{k-1}( x ) \} f(x)\, dx.$$
For $s \in \II$ one has strict positivity of  $F(s+\epsilon)$,
so by the induction hypothesis  $\HH_{k-1}$ we have
$v_k( s + \epsilon) - v_k (s) \leq F(s+\epsilon)\left\{ v_{k-1}( s + \epsilon)  - v_{k-1}( s ) \right\}<0$.
\end{proof}

\subsection*{\sc Optimal Threshold Functions}\subsectionnewline

The monotonicity of $v_{k-1}$ tells us that
the integrand in \eqref{eq:Bellman-general-f} equals
the right maximand $\{1 + v_{k-1}(x)\}$ on a certain initial segment of $[s, \infty)$, and
it equals the left maximand $v_{k-1}(s)$ on the rest of the segment. This observation leads to
a useful reformulation of the Bellman recursion; specifically, if we set
\begin{equation}\label{eq:optimal-threshold-general-f}
h_{k} (s) = \sup \{ x \in [s , \infty) : F(x) < 1 \text{ and } v_{k-1}(s)\leq 1 + v_{k-1}(x) \},
\end{equation}
then the Bellman recursion \eqref{eq:Bellman-general-f} can be written as
\begin{equation}\label{eq:Bellman-general-f-h-explicit}
v_k(s) = \{ 1- F( h_{k} (s) ) + F(s) \} v_{k-1}(s) + \int_s^{h_{k} (s)} \{ 1 + v_{k-1}(x) \} \, dF(x).
\end{equation}
The functions $\{ h_k: 1 \leq k < \infty\}$ defined by \eqref{eq:optimal-threshold-general-f}
are called the \emph{optimal threshold functions}.

If $v_{k-1}(s) \leq 1$, the characterization \eqref{eq:optimal-threshold-general-f}
has an informative policy interpretation.
Namely, if $v_{k-1}(s) \leq 1,$ then the optimal strategy for the decision
maker is the greedy strategy where one accepts any arriving observation that is as large as $s$.
On the other hand, if $ v_{k-1} (s) > 1$,
the optimal decision maker needs to act more conservatively; when $k$ observations remain to be seen,
one only accepts the newly arriving
observation if it falls in the interval $[s, h_{k}(s)]$.

When $F$ is admissible, we have the strict monotonicity of $v_{k-1}$, and this
allows a second  characterization of the threshold function:
\begin{equation} \label{eq:hk-optimally-conservative}
h_k ( s )  \text{ uniquely satisfies } v_{k-1}(s) = 1 + v_{k-1} ( h_k ( s ) ) \quad \quad \text{ if } v_{k-1} (s) > 1.
\end{equation}
The value  $h_{k}(s)$ of the threshold function thus marks
the point of indifference between the optimal acceptance region and the optimal rejection region.
The characterization \eqref{eq:hk-optimally-conservative} also motivates a definition.
\begin{definition}[Critical Value]\label{def:criticalvalue}
If $F$ is admissible, then the unique solution of the equation
$
v_k(s)=1
$
is called the \emph{critical value}, and it is denoted by $s_k^*$.
\end{definition}

The analytical character of  $h_k$ changes at $s_k^*$,
and one has to be attentive to the  differing  behavior
of $h_k$ above and below $s_k^*$.
We will not need this distinction until Section \ref{se:smmoothness-of-value-functions}, but
it is critical there.

We complete this section by recording  two simple (but useful) bounds on the time-difference
of the value function. These bounds follow from the characterization \eqref{eq:optimal-threshold-general-f}
for the optimal threshold $h_k$ and the monotonicity of the
value function $v_{k-1}$.

\begin{lemma}\label{lm:Bellman-bounds}
For $s \geq 0$ and $1\leq k < \infty $, we have the inequalities
\begin{equation}\label{eq:simple-difference-bounds}
0 \leq v_{k}(s)-v_{k-1}(s)\leq  F( h_k(s) ) - F(s) \leq 1.
\end{equation}
\end{lemma}

From a modeler's perspective, this inequality is intuitive since $F(h_k( s )) - F(s)$
can be interpreted as the probability that one selects the next observation when
$k$ observations remain to be seen. A formal confirmation of \eqref{eq:simple-difference-bounds}
illustrates the handiness of the
second form \eqref{eq:Bellman-general-f-h-explicit} of the Bellman equation.

\begin{proof}[Proof of Lemma \ref{lm:Bellman-bounds}]
First, note that after subtracting $v_{k-1}(s)$ from both sides of equation \eqref{eq:Bellman-general-f-h-explicit},
we have
\begin{equation*}
    v_k(s)-v_{k-1}(s) = \int^{h_k(s)}_{s}  \{ 1+v_{k-1}(x)-v_{k-1}(s) \} \,dF(x).
\end{equation*}
The map $x \mapsto v_{k-1}(x)$ is monotone decreasing,
so the factor $\{ 1+v_{k-1}(x)-v_{k-1}(s) \}$ is bounded above by one.
This gives us our upper bound in \eqref{eq:simple-difference-bounds}.
The representation \eqref{eq:optimal-threshold-general-f} for $h_k$
tells us the integrand is non-negative on $[s,h_k(s)]$, and this gives the lower bound in \eqref{eq:simple-difference-bounds}.
\end{proof}

\subsection*{\sc Value Function Submodularity} \subsectionnewline

If one increases the number $k$ of observations yet to be seen, then the decision maker faces a richer set of
future possibilities. This in turn suggests that the decision maker may want to act more conservatively, keeping more
powder dry for future action.
Specifically, one might guess that $h_{k+1}(s)\leq h_k(s)$
for all $s \in [0, \infty)$ and all $1 \leq k < \infty$.
We confirm this guess as a corollary of the next proposition which gives us a pivotally useful
property of the value functions.

\begin{proposition}[Submodularity of the Value Functions]\label{pr:value-functions-submodularity}
The sequence of value functions $\{v_k : 1 \leq k < \infty\}$
determined by the Bellman recursion \eqref{eq:Bellman-general-f}
is \emph{submodular} in the sense that
for all $1 \leq k < \infty$ one has
\begin{equation}\label{eq:submodularity}
    v_{k-1} (s) - v_{k-1}(t) \leq v_{k} ( s ) - v_{k} ( t ) \quad \quad \text{for all } 0 \leq s \leq t < \infty.
\end{equation}
\end{proposition}

\begin{proof}
We first derive a recursion for the difference $v_{k} ( s ) - v_{k}( t )$.
For $0 \leq s \leq t < \infty$, we have from \eqref{eq:Bellman-general-f} that
\begin{align}\label{eq:max-integral}
v_{k} ( s ) & \!-\! v_{k}( t ) \!=\! F(s) \{ v_{k-1}( s ) \!-\! v_{k-1}( t ) \} \\
& \!+\! \int_s^t \max \{ v_{k-1} ( s ) \!-\! v_{k-1} ( t ) , 1 \!+\! v_{k-1} ( x ) \!-\! v_{k-1} ( t ) \}  \, dF(x) \nonumber \\
& \!+\! \int_t^\infty [ \max \{ v_{k-1} ( s ) , 1 \!+\! v_{k-1} ( x ) \} \!-\! \max \{ v_{k-1} ( t ) , 1 \!+\! v_{k-1}( x ) \} ]  \, dF(x). \nonumber
\end{align}
Next, we let
\begin{align*}
  a_{k-1}(x) &\equalbydef \min \{ v_{k-1}(s) - v_{k-1}(t), v_{k-1}(s) - v_{k-1}(x) - 1 \}, \\
  b_{k-1}(x) &\equalbydef \min \{ 1 + v_{k-1}(x) - v_{k-1}(t) , 0 \},
\end{align*}
and we note that the difference
$$ \max \{ v_{k-1} ( s ) , 1 + v_{k-1} ( x ) \} - \max \{ v_{k-1} ( t ) , 1 + v_{k-1}( x ) \}$$
which appears in the last integrand of \eqref{eq:max-integral}
can be written as
$$
\max \{ v_{k-1} ( s ) , 1 + v_{k-1} ( x ) \} - \max \{ v_{k-1} ( t ) , 1 + v_{k-1}( x ) \} = \max \{a_{k-1}(x), b_{k-1}(x) \}.
$$
Here $s \leq t$, so when $b_{k-1}(x) < 0$
the monotonicity of the value functions in Lemma \ref{lm:value-function-decreasing-general-f}
implies that $ 0 \leq v_{k-1} (s) - v_{k-1} (x) - 1 $.
It then follows that $ 0 \leq a_{k-1}(x)$ and
$\max \{a_{k-1}(x), b_{k-1}(x) \} = \max \{a_{k-1}(x), 0 \}$.
In general, we then have the equivalence
$$
\max \{ v_{k-1} ( s ) , 1 + v_{k-1} ( x ) \} - \max \{ v_{k-1} ( t ) , 1 + v_{k-1}( x ) \} = \max \{a_{k-1}(x), 0 \},
$$
and we can substitute this representation and the explicit expression for $a_{k-1}(x)$
in \eqref{eq:max-integral} to obtain the simplified \emph{difference recursion}
\begin{align}\label{eq:difference-bellman}
v_{k} ( s ) \!-\! v_{k}( t ) &= F(s) \{ v_{k-1}( s ) \!-\! v_{k-1}( t ) \} \\
+  \int_s^t & \max \{ v_{k-1} ( s ) \!-\! v_{k-1} ( t ) , 1 \!+\! v_{k-1} ( x ) \!-\! v_{k-1} ( t ) \} \, dF(x) \nonumber \\
+ \int_t^\infty & \max \{ \min \{ v_{k-1} (s) - v_{k-1} (t), v_{k-1} (s) - v_{k-1} (x) - 1 \} , 0 \} \, dF(x). \nonumber
\end{align}

We now let $\HH_k$ be the assertion that
\begin{equation*}
v_{ k - 1 } ( s ) - v_{ k - 1 }( t ) \le v_{ k }( s ) - v_{ k }( t )
\quad \quad \text{ for all } 0 \leq s \leq t < \infty,
\end{equation*}
and we prove by induction that $\HH_k$ holds for all $k\geq 1$. We first note that
for $k = 1$ we have $ v_0 ( s ) = 0 $ for all $s \in [0, \infty)$.
By the difference recursion \eqref{eq:difference-bellman} we obtain
$v_1 ( s ) - v_1 ( t ) =  F(t) - F(s)\geq 0 = v_0 ( s ) - v_0 ( t ) $, so the base case $\HH_1$ holds.

Next, we suppose that $\HH_{k-1}$ holds, and we
apply $\HH_{k-1}$ to \emph{all} of the terms on the right-hand side of \eqref{eq:difference-bellman}.
We then obtain that
\begin{align}\label{eq:difference-bellman}
v_{k} ( s ) \!-\! v_{k}( t ) \leq {} &   F(s) \{ v_{k}( s ) \!-\! v_{k}( t ) \} \nonumber\\
 & + \int_s^t  \max \{ v_{k} ( s ) \!-\! v_{k} ( t ) , 1 \!+\! v_{k} ( x ) \!-\! v_{k} ( t ) \} \, dF(x) \nonumber \\
 & + \int_t^\infty  \max \{ \min \{ v_{k} (s) - v_{k} (t), v_{k} (s) - v_{k} (x) - 1 \} , 0 \} \, dF(x). \nonumber
\end{align}

We can now apply the difference recursion \eqref{eq:difference-bellman} a second time after we replace $k$ by $k+1$.
This tells us that the right-hand side above is equal to the difference $v_{ k + 1 } ( s ) - v_{ k + 1}( t )$,
thus completing the proof of $\HH_k$ and of the proposition.
\end{proof}

The submodularity guaranteed by  Proposition \ref{pr:value-functions-submodularity} is more powerful than one might expect.
In particular, it delivers three basic corollaries.

\begin{corollary}[Monotonicity of Optimal Thresholds]\label{cor:monotonicity-optimal-thresholds}
For the threshold functions characterized by \eqref{eq:optimal-threshold-general-f} we have
\begin{equation}\label{eq:optimal-threshold-monotonicity}
h_{k+1} (s) \leq h_k(s) \quad \quad \text{ for }  0 \leq s < \infty.
\end{equation}
\end{corollary}

\begin{proof}
Here we only have to note that
\begin{align*}
h_{ k + 1 } (s) & = \sup \{ x \in [s, \infty)  : F(x) < 1 \text{ and } v_k (s) - v_k (x) \le 1 \}\\
& \le \sup \{ x \in [s, \infty)  : F(x) < 1 \text{ and } v_{ k - 1 } (s) - v_{ k - 1 } (x) \le 1 \}\\
& = h_{ k } (s),
\end{align*}
where the one inequality comes directly from the submodularity \eqref{eq:submodularity}
and the two equalities
come from \eqref{eq:optimal-threshold-general-f}.
\end{proof}

\begin{corollary}[Concavity in $k$ of the Value Functions.]\label{cor:value-function-concave-n}
The value functions are concave as functions of $k$; that is, for each  $s \in [0,\infty)$  and all  $k \geq 1$, one has
\begin{equation*}
v_{k+1}(s) - 2 v_{k}(s) + v_{k-1}(s) \leq 0.
\end{equation*}
\end{corollary}

\begin{proof}
By the monotonicity \eqref{eq:optimal-threshold-monotonicity}
of the optimal threshold functions, the recursion \eqref{eq:Bellman-general-f-h-explicit}
gives us the difference identity
\begin{align*}
v_{k+1}(s) - v_{k}(s)
 = {} & v_{k}(s) - v_{k-1}(s) \\
& + \int_s^{h_{k+1}(s)} \{ v_{k-1}(s) -  v_{k-1}(x) - v_{k}(s) + v_{k}(x)  \} \, dF(x)\\
& + \int_{h_{k+1} (s)}^{h_k (s)} \{  v_{k-1}(s) - 1 - v_{k-1}(x) \} \, dF(x),
\end{align*}
and it suffices to check that the two integrands on the right-hand side are non-positive.
Non-positivity of the first integrand follows from the submodularity \eqref{eq:submodularity},
and non-positivity of the second integrand follows from the characterization of $h_k(s)$ in \eqref{eq:optimal-threshold-general-f}.
\end{proof}

\begin{corollary}[Concavity in $n$ of the Expected Length]\label{cor:mean-concave-n}
For any continuous $F$, the map $n \mapsto \E[L_n(\pi^*_n)]$ is concave in $n$.
\end{corollary}

This is just a special case of Corollary \ref{cor:value-function-concave-n}
(where one just takes $s=0$ and $k = n$), but,
as we will see in Section \ref{se:Mean-DePoissonization}, this concavity
carries noteworthy force.

\begin{remark}[Further Context: an Offline Open Problem]\label{rem:concavity-n-offline}
It is not known if the corresponding concavity holds for the \emph{offline} monotone subsequence problem.
That is, we do not know if the map $n \mapsto \E[L_n]$ is concave where $L_n$ is defined by \eqref{eq:LnOffLineDef}.
In this case we do know $\E [L_n] = 2 n^{1/2} - \alpha n^{1/6} + o( n^{1/6})$ so concavity does seem like a
highly plausible conjecture.
\end{remark}

\section{\sc Intermezzo: Possibilities for De-Poissonization}\label{se:Mean-DePoissonization}

If $N$ is an integer valued random variable,
then one can consider the problem of sequential selection of a monotone increasing subsequence
from the random length sequence $\S=\{X_1, X_2, \ldots, X_N\}$.
Here, as usual, the elements of the sequence are independent
with a common continuous distribution $F$, and they are also independent of $N$.
We also assume that the decision maker knows $F$ and the distribution of $N$,
but the decision maker does not know the value of $N$ until the
sequence $\S$ has been exhausted. We let $L_{N} (\pi)$ denote the number
of selections that are made when one follows a
policy $\pi$ for sequential selection from $\S$.

\begin{proposition}[Information Lower Bound]\label{prop:info-lower-bound}
If $\E[N]=n$ for some $n \in \N$, then
\begin{equation}\label{eq:info-lower-bound}
\E [L_{N} (\pi)] \leq \E[L_n(\pi_n^*)].
\end{equation}
\end{proposition}

\noindent\emph{Proof.} 
The policy $\pi$ is determined before the realization of $N$ is known, and,
for any given $j$, the policy $\pi$ is suboptimal when
it is used for sequential selection from the sequence $\{X_1,X_2, \ldots, X_j\}$.
Thus, if we condition on $N=j$, we then have
\begin{equation}\label{eq:suboptimalBD}
E[L_N(\pi) \, | \, N=j] \leq \E[L_j(\pi^*_j)].
\end{equation}
Now, if we take  $\phi:[0,\infty) \rightarrow [0,\infty)$ to be the piecewise linear extension of the map $j \mapsto \E[L_j(\pi^*_j)]$,
then by Corollary \ref{cor:mean-concave-n} we have that $\phi$ is also concave.
Finally, by the suboptimality \eqref{eq:suboptimalBD},  the definition of $\phi$,
and Jensen's inequality we obtain
\[
\pushQED{\qed}
\E[L_N(\pi)]  \leq \sum_{j=0}^\infty \E[L_j(\pi^*_j)] \P(N=j) = \E[ \phi (N)]  \leq \phi( \E [N] ) =\E[L_n(\pi_n^*)]. \qedhere
\popQED
\]

The next corollary establishes one of the five assertions of Theorem \ref{thm:OurCLT}. It is an immediate consequence of
Proposition \ref{prop:info-lower-bound} and the lower half of the mean bound \eqref{eq:BD-optimal-mean}
from \citeasnoun{BruDel:SPA2001}.

\begin{corollary}\label{cor:mean-lower-bound}
For any continuous $F$ we have as $n \rightarrow \infty$ that
\begin{equation}\label{eq:mean-bounds-main-theorem-deux}
(2n)^{1/2} - O(\log n) \leq E[L_n(\pi^*_n)].
\end{equation}
\end{corollary}
This is a notable improvement
over the bound \eqref{eq:RheeTalagrandLowerBound} that had been established by several earlier
investigations; it improves a $O(n^{1/4})$ error bound all the way down to $O(\log n)$. For the central limit theorem \eqref{eq:TheCLT},
one could still get along with a lower bound as weak as  $(2n)^{1/2} - o(n^{1/4})$.

\subsection*{\sc De-Poissonization and Decision Problems}\subsectionnewline

We get the bound \eqref{eq:mean-bounds-main-theorem-deux} by a \emph{de-Poissonization argument} in the sense that
a ``fixed $n$" fact is extracted from a ``Poisson $N$" fact. Such arguments are common in computer science,
combinatorics and analysis; one finds many examples in
\citeasnoun{JacSzp:TCS1998},
\citeasnoun[Subsection VIII.5.3]{FlaSed:CUP2009}, and \citeasnoun[Chapter 6]{Kor:SPRI2004}.
Still, Proposition \ref{prop:info-lower-bound} is our only instance of a de-Poissonization argument,
and the proof of the proposition suggests in part why one may be hard-pressed to find more.

Decision problems are unlike the classical examples mentioned above. The Poisson $N$ problem and the fixed $n$ problem have different
optimal policies, and this mismatch forestalls the kind of direct analytical connection one has in the classical examples.
Conditioning on $N=j$ does engage the problem, but the suboptimality of the mismatched policy leads only to one-sided
relations such as \eqref{eq:info-lower-bound} and \eqref{eq:suboptimalBD}.

\section{Smoothness of the Value and Threshold Functions}\label{se:smmoothness-of-value-functions}

We need to show that the value functions
associated with an admissible distribution $F$ are continuously differentiable on $\II$.
As preliminary step, we consider the differentiability of the threshold functions
in a region determined by the critical values $s_k^*$ that were defined in Section \ref{se:DP-general-f}.

\begin{lemma}[Differentiability of the Threshold Functions]\label{lm:hk-differentiable}
Take $F$ to be admissible and take $k>1$. If $v_{k-1}$ is differentiable on $\II$ and
$s \in \II \cap [0, s_{k-1}^*) $, then
$h_k$ is differentiable at $s$, and one has
\begin{equation}\label{eq:hk-first-derivative}
h_k'(s) = \frac{v_{k-1}'(s)}{v_{k-1}'(h_k(s))} \geq 0.
\end{equation}
\end{lemma}
\begin{proof}
Set $Q(x,y) = - 1 + v_{k-1}(x) - v_{k-1}(y)$.
If $Q_y$ denotes the partial derivative of the function $Q$ with respect to $y$,
we know by our hypotheses that $Q_y$ exists,
and Lemma \ref{lm:value-function-decreasing-general-f} implies that the partial derivative $Q_y$ is strictly positive.
Now, if $(x_0,y_0)$ satisfies  $Q(x_0,y_0)=0$, then by the
implicit function theorem there is a neighborhood $\mathcal{N}_0$ of $x_0$ where one can solve
$Q(x,y)=0$ uniquely for $y$, and the solution $y$ is a differentiable function of $x$ for all $x \in \mathcal{N}_0$.
Moreover, if $x < s_{k-1}^*$ then \eqref{eq:hk-optimally-conservative} tells us $Q(x,y)=0$ if and only if $y=h_k(x)$,
so $h_k$ is differentiable as claimed. Given the differentiability of $h_k$ at $s$,
the formula \eqref{eq:hk-first-derivative} follows directly from
$v_{k-1}(x) =1 + v_{k-1}(h_k(x))$ by differentiation and the chain rule.
The non-negativity of $h_k'(s)$ then follows because the value function $v_{k-1}$
is strictly decreasing.
\end{proof}

\begin{proposition}[Continuous Differentiability of the Value Functions]\label{prop:vk-differentiable-w-continuous-derivative}
If $F$ is admissible,
then for each $1 \leq k < \infty$  the value function
$s \mapsto v_k(s)$ is continuously differentiable on $\II$, and we have
\begin{equation}\label{eq:first-derivative-recursion-general-f}
v_k'(s) = -f(s) + \{ 1 - F(h_k(s)) + F(s) \} \, v_{k-1}'(s) \quad \text{for } s \in \II.
\end{equation}
\end{proposition}

\begin{proof}
We argue by induction on $k$, and we first note
for $k=1$ that $v_1(s) = 1 - F(s)$, so
$v_1'(s) = -f(s)$ and \eqref{eq:first-derivative-recursion-general-f} holds since $v_0(s)\equiv 0$.
Next, we assume by induction that $v_{k-1}$ is continuously differentiable on $\II$.
If $s < s^*_{k-1}$  then the induction assumption and Lemma \ref{lm:hk-differentiable}
imply that $h_{k}$ is differentiable at $s$.
We then differentiate \eqref{eq:Bellman-general-f-h-explicit} to find
\begin{align*}
v_k'(s)  = {}  & -  f(s) + \{ 1 - F(h_k(s)) + F(s) \} \, v_{k-1}'(s) \\
               & + f(h_k(s)) \{ 1 - v_{k-1}(s) + v_{k-1}(h_k(s)) \} h_k'(s) \nonumber \\
         = {}  & -  f(s) + \{ 1 - F(h_k(s)) + F(s) \} \, v_{k-1}'(s), \nonumber
\end{align*}
where the last step used the characterization \eqref{eq:hk-optimally-conservative} of $h_k$.
Alternatively, if $s > s^*_{k-1}$ we have $F(h_k(s)) = 1$ and \eqref{eq:Bellman-general-f-h-explicit} says simply that
$$
v_k( s )   = F( s ) v_{ k - 1 }( s ) +
\int_s^\infty \{ 1 + v_{ k - 1 }( x ) \}  f( x ) \, dx.
$$
Differentiation of this integral then gives us \eqref{eq:first-derivative-recursion-general-f}.
Thus, one has  that \eqref{eq:first-derivative-recursion-general-f} holds on all of
$\II_k = \II \setminus \{s^*_{k-1} \}$. Moreover, taking left and right limits in \eqref{eq:first-derivative-recursion-general-f} gives us
$$
\lim_{s \nearrow s^*_{k-1}} v_{k}'(s) = -f( s^*_{k-1} ) +  F(s^*_{k-1}) v_{k-1}'(s^*_{k-1}) = \lim_{s \searrow s^*_{k-1}} v_{k}'(s).
$$
It is almost obvious that these relations imply
the continuous differentiability $v_k$,
but to make it crystal clear let $\gamma$ be
the common value of the limits above and define a continuous function $\bar v: \II \rightarrow \R$ by setting
$$
\bar v(s) =
\begin{cases}
v_k'(s) &\quad \text{if } s < s^*_{k-1} \\
\gamma &\quad \text{if } s = s^*_{k-1} \\
v_k'(s) &\quad \text{if } s > s^*_{k-1}.
\end{cases}
$$
Next, we obtain by piecewise integration that
$$
v_k(s)= v_k(0) + \int_0^s \bar v(u) \, du \quad \text{for all } s \in \II,
$$
implying, as expected, that $v_k$ is continuously differentiable on $\II$.
\end{proof}

\section{Spending Symmetry: Curvature of the Value Functions }\label{se:DP-special-properties}

For any continuous $F$ the distribution of $L_n(\pi_n^*)$ is the same; this is an invariance property --- or a \emph{symmetry}.
When one chooses a particular $F$, say the uniform distribution, there is a sense in which one \emph{spends symmetry}.

All earlier analyses  of $L_n(\pi_n^*)$ passed
directly to the uniform distribution without any apparent thought
about what might be lost or gained by the transition. Still, it does make a difference how one spends this symmetry.
The distribution of $L_n(\pi_n^*)$ is insensitive to $F$, but the value functions are not.

Specifically, for the uniform distribution the value functions are concave, but for the exponential distribution the value functions are convex.
This change of curvature gives one access to different estimates. Over the next several sections we see how specialization of the
driving distribution has a big influence on the estimation of variances and conditional variances.

\subsection*{\sc Concavity of the Value Functions in the Uniform Model}\subsectionnewline

We first break symmetry in the conventional way and take $F$ to be the uniform distribution on $[0,1]$.
Specialization of the Bellman recursion \eqref{eq:Bellman-general-f}
defines the sequence of value functions $\{\vu_k : 1 \leq k < \infty\}$, and specialization of the characterization
\eqref{eq:optimal-threshold-general-f} defines the sequence of
threshold functions $\{\hu_k: 1 \leq k < \infty\}$.
Here, we have by \eqref{eq:optimal-threshold-general-f} that $\hu_k(s) \leq 1$ for all $s\in[0,1]$ and $1 \leq k <\infty$.

\begin{lemma}[Concavity of the Uniform Value Functions]\label{lm:value-function-concave-uniform}
For each $1 \leq k < \infty$
the value function $\vu_k : [0,1] \rightarrow \R^+$
is concave.
\end{lemma}

\begin{proof}
Proposition \ref{prop:vk-differentiable-w-continuous-derivative}
tells us  $\vu_k$  is continuously differentiable
on $(0,1)$, and we prove concavity by showing that $s \mapsto (\vu_k)'(s)$ non-increasing on $(0,1)$.
We let $\HH_k$ be the assertion
$$
(\vu_k)'(s + \epsilon) \leq (\vu_k)'(s) \quad \quad
\text{for all } s \in (0,1) \, \text{and } 0 < \epsilon < 1-s,
$$
and we argue by induction. For $k = 1$ we have $\vu_1(s) = 1-s$, so $(\vu_1)'(s) = - 1$
and $\HH_1$ holds trivially.

Now, if we specialize the derivative recursion \eqref{eq:first-derivative-recursion-general-f}
to the uniform model we have
\begin{equation*}
(\vu_{k})'(s) = -1 + \{ 1 - \hu_k(s) + s \} ( \vu_{k-1})'(s), \quad \quad \text{for } s \in (0,1),
\end{equation*}
so if we assume that $\HH_{k-1}$ holds then we have
\begin{align}\label{eq:first-derivative-difference-uniform}
(\vu_k)'(s + \epsilon) - (\vu_k)'(s)
 = {} & \{ 1 - \hu_k(s) + s \} \{ (\vu_{k-1})'(s + \epsilon) - (\vu_{k-1})'(s)  \} \\
& + \{  \hu_k(s)  - \hu_k(s + \epsilon)  + \epsilon \} (\vu_{k-1})'(s + \epsilon). \nonumber
\end{align}
Since $0 \leq s \leq \hu_k(s) \leq 1$, we see from $\HH_{k-1}$
that the first summand on the right-hand side
of \eqref{eq:first-derivative-difference-uniform} is non-positive.
Monotonicity of $\vu_k$ also tells us
$(\vu_{k-1})'(s + \epsilon) \leq 0$,
so to complete the induction step we just need to check that
\begin{equation}\label{eq:g-first-derivatives-uniform}
g(s, \epsilon) \equalbydef \hu_k(s + \epsilon) -  \hu_k(s)   \leq  \epsilon.
\end{equation}
From the definition of the critical value $s^*_{k-1}$ we have
\begin{equation*}
g(s, \epsilon) =
\begin{cases}
    \hu_k(s + \epsilon) - \hu_k(s ) & \text{ if } s < s+\epsilon < s^*_{k-1} \\
    1 -  \hu_k(s )                  & \text{ if } s < s^*_{k-1} \leq s+\epsilon \\
    0                               & \text{ if } s^*_{k-1} \leq s < s+\epsilon,
\end{cases}
\end{equation*}
so we only need to check \eqref{eq:g-first-derivatives-uniform} in the first two cases.

For $s < s+\epsilon < s^*_{k-1}$, we know by Lemma \ref{lm:hk-differentiable} that
$\hu_k(s )$ is differentiable at $s$, so by the
induction assumption $\HH_{k-1}$ and the negativity of $(\vu_{k-1})'(s)$
we have
$$
0 \leq (\hu_k)'(s ) \leq 1 \quad \quad \text{for all } s \in (0, s^*_{k-1}).
$$
Thus, $\hu_k$ is Lipschitz-1 continuous on $(0, s^*_{k-1})$,  and we have
\begin{equation*}
g(s, \epsilon) = \hu_k(s + \epsilon) - \hu_k(s )  \leq \epsilon \quad \quad \text{ for all } s < s+\epsilon < s^*_{k-1}.
\end{equation*}
For the second case where $s < s^*_{k-1} \leq s+\epsilon$, we first note that $\hu_k(s^*_{k-1})=1$ and that $\hu_k$ is continuous, so
we have
\begin{equation*}
g(s, \epsilon)
= \lim_{u \nearrow s^*_{k-1}} \{ \hu_k(u) - \hu_k(s ) \}
\leq \lim_{u \nearrow s^*_{k-1}} \{ u - s \} = s^*_{k-1} - s < \epsilon,
\end{equation*}
where the inequality follows from the Lipschitz-$1$ property of $\hu_k$. This completes the second check and the proof of the
induction step.
\end{proof}

\subsection*{\sc Convexity of the Value Functions in the Exponential Model}\subsectionnewline \label{se:convexity-exponential}

We now break the symmetry in a second way.
We take $F(x) = 1 - e^{-x}$ for $x \geq 0$, and we let $\ve_k$ and $\he_k$ denote the corresponding
value and threshold functions. We will shortly find that $\ve_k$ is convex on $[0, \infty)$
for all $k\geq 1$, but we need a preliminary lemma.

\begin{lemma}\label{lm:first-derivative-LB-exponential}
For $1 \leq k <\infty$ and $s \in (0, \infty)$ one has
\begin{equation}\label{eq:first-derivative-exponential-lower-bound}
 - \{ 1 - e^{ - \he_{k+1}(s) + s  } \}^{-1} \leq (\ve_k)'(s).
\end{equation}
\end{lemma}

\begin{proof}
By Proposition \ref{prop:vk-differentiable-w-continuous-derivative} we know
that $\ve_k$ is continuously differentiable
and by \eqref{eq:first-derivative-recursion-general-f} we have
\begin{equation}\label{eq:first-derivative-exponential}
(\ve_k)'(s) = ( 1 - e^{-s} + e^{-\he_k(s)} ) (\ve_{k-1})'(s) - e^{-s}.
\end{equation}
We now let  $\HH_k$ be the assertion that
$$
 - \{ 1 - e^{ - \he_{k+1}(s) + s  } \}^{-1} \leq (\ve_k)'(s), \quad \quad \text{for all } s \in (0, \infty),
$$
and we argue by induction.
For $k = 1$ we have $\ve_1(s) = e^{-s}<1$,  so by \eqref{eq:optimal-threshold-general-f}
we have $\he_2(s) = \infty$. In turn this  gives us
$$
- \{ 1 - e^{ - \he_2(s) + s } \}^{-1} = -1 \leq  - e^{-s} = (\ve_1)'(s),
$$
which verifies $\HH_1$.

Next, if we assume that $\HH_{k-1}$ holds and we substitute the lower bound from $\HH_{k-1}$
into \eqref{eq:first-derivative-exponential}, then rearrangement gives us
$$
- \{ 1 - e^{ - \he_{k}(s) + s } \}^{-1} [1 - e^{-s} \{1 - e^{ - \he_k(s) + s } \} ] - e^{-s} \leq (\ve_{k})'(s).
$$
From \eqref{eq:optimal-threshold-monotonicity} we have $\he_{k+1}(s) \leq \he_k(s)$, so we now
have
$$
- \{ 1 - e^{ - \he_{k+1}(s) + s } \}^{-1} \leq - \{ 1 - e^{ - \he_{k}(s) + s } \}^{-1} \leq (\ve_{k})'(s),
$$
and this is just what one needs to complete the induction step.
\end{proof}

We now have the main result of this section.

\begin{lemma}[Convexity of the Exponential Value Functions]\label{lm:value-functions-exponential-convexity}
For each $1 \leq k < \infty$,
the value function $\ve_k : [0, \infty) \rightarrow \R^+$
is convex on $[0, \infty)$.
\end{lemma}

\begin{proof}
By Proposition \ref{prop:vk-differentiable-w-continuous-derivative}
we know that $\ve_k$
is continuously differentiable, and we again argue by induction. This time
we take $\HH_k$ to be the assertion
$$
(\ve_k)'(s) \leq (\ve_k)'(s + \epsilon) \quad \quad \text{for all } s \in (0, \infty) \text{ and } \epsilon>0.
$$
For $k = 1$, we have $\ve_1(s) = e^{-s}$ and $(\ve_1)'(s) = - e^{-s}$ so the base case $\HH_1$ of the induction
is valid.

Now, by \eqref{eq:first-derivative-exponential} applied twice we have
\begin{align}\label{eq:first-derivative-difference-exponential}
(\ve_k)'(s) \!-\! (\ve_k)'(s + \epsilon)
& = [1 \!-\! e^{-s - \epsilon} \!+\! e^{-\he_k(s + \epsilon)}]\{ (\ve_{k-1})'(s) \!-\! (\ve_{k-1})'(s + \epsilon)\}\\
& + \{ e^{-\he_k(s)}[ 1 \!-\! e^{-\he_k(s + \epsilon) + \he_k( s )}] \!-\! e^{-s}[ 1 \!-\! e^{-\epsilon}] \}  (\ve_{k-1})'(s) \nonumber\\
& - e^{-s}[1 \!-\! e^{-\epsilon}].\nonumber
\end{align}
The induction hypothesis $\HH_{k-1}$, tells us that $s \mapsto \ve_{k-1}(s)$ is convex,
so by \eqref{eq:hk-first-derivative} we have $(\he_k)'(s) \geq 1$ for $s \in (0, \infty)$, and
this gives us the bound
$$
- \he_k(s + \epsilon) + \he_k( s )  \leq - \epsilon.
$$
We always have $s \leq \he_k(s)$ and $(\ve_{k-1})'(s) \leq 0$,
so \eqref{eq:first-derivative-difference-exponential} implies the simpler bound
\begin{align*}
(\ve_k)'(s) - (\ve_k)'(s + \epsilon)
\leq {} & [1 - e^{-s - \epsilon} + e^{-\he_k(s + \epsilon)}]\{  (\ve_{k-1})'(s) - (\ve_{k-1})'(s + \epsilon) \}\\
& -  e^{-s}[ 1 - e^{-\epsilon}][1 -  e^{-\he_k(s) + s }]  (\ve_{k-1})'(s) \nonumber\\
& -  e^{-s}[1 - e^{-\epsilon}].\nonumber
\end{align*}
We only need to check that this bound is non-positive. By the induction hypothesis $\HH_{k-1}$
and $ s + \epsilon \leq \he_k(s + \epsilon)$,  we see the first term is non-positive.
The bound \eqref{eq:first-derivative-exponential-lower-bound} tells us
$$
-  [1 -  e^{-\he_k(s) + s }]  (\ve_{k-1})'(s) \leq 1,
$$
so, when we replace $-  [1 -  e^{-\he_k(s) + s }]  (\ve_{k-1})'(s)$ with its
upper bound, we also see that the second and the third terms sum to zero. This completes the proof of the induction step
and of the lemma.
\end{proof}

\section{Martingale Relations and  $L_n(\pi^*_n)$}\label{se:probabilistic-interpretation-and-Optimality-martingale}

One can represent  $L_n(\pi^*_n)$ as a sum of functionals of a time non-homogeneous Markov chain.
To see how this goes, we first set $ M_0 = 0 $
and then we define $M_i$ recursively by
\begin{equation}\label{eq:running-maximum-general-f}
M_i =
\begin{cases}
	M_{ i - 1 } & \text{ if } X_i \not\in [ M_{ i - 1 }, h_{ n - i + 1 }( M_{ i - 1 } ) ] \\
	X_i & \text{ if } X_i \in [ M_{ i - 1 }, h_{ n - i + 1 }( M_{ i - 1 } ) ],
\end{cases}
\end{equation}
so, less formally, $ M_i $ is the maximum value of the elements
of the subsequence that have been selected up to and including time $ i $.
Since we accept $X_i$ if and only if $X_i \in [ M_{ i - 1 }, h_{ n - i + 1 }( M_{ i - 1 } ) ]$
and since $L_n(\pi^*_n)$ counts the number of
the observations $X_1, X_2, \ldots, X_n$ that we accept, we have
\begin{equation}\label{eq:Ln}
L_n(\pi^*_n) = \sum_{i=1}^n \1 ( X_i \in  [ M_{ i - 1 }, h_{ n - i + 1 }( M_{ i - 1 } ) ] ).
\end{equation}
It is also useful to set $L_0(\pi^*_n) = 0 $ and to introduce the shorthand,
\begin{equation*}
L_i(\pi^*_n) \equalbydef
\sum_{j=1}^i \1 ( X_j \in  [ M_{ j - 1 }, h_{ n - j + 1 }( M_{ j - 1 } ) ] ), \quad \quad \text{for } 1 \leq i \leq n.
\end{equation*}
We now come to a martingale that is central to the rest of our analysis.

\begin{proposition}[Optimality Martingale]\label{pr:Optimality-martingale}
The process  $\{ Y_i: i=0,1, \ldots, n \}$  defined by setting
\begin{equation}\label{eq:Optimality-martingale}
Y_i = L_i (\pi_n^*) + v_{ n - i }( M_i ) \quad \quad for \ 0 \le i \le n,
\end{equation}
is a martingale with respect to the filtration
$ \F_i = \sigma \{ X_1, X_2, \ldots, X_i \}$, $ 1 \le i \le n $.
\end{proposition}

\noindent\emph{Proof.} 
Obviously $ Y_i $ is $ \F_i$-measurable and bounded. Moreover, by the definition of $ v_{ n - i }( s ) $ we have
$ v_{ n - i }( M_i ) = \E  [ L_n ( \pi_n^* ) - L_i( \pi_n^* ) \ | \ \F_{ i } ] $,
so
\[
\pushQED{\qed}
Y_i = L_i( \pi_n^* ) + \E  [ L_n ( \pi_n^* ) - L_i( \pi_n^* ) \ | \ \F_{ i } ] = \E  [ L_n ( \pi_n^* ) \ | \ \F_{ i } ].\qedhere
\popQED
\]

Since the martingale $\{Y_i: 1 \leq i \leq n\}$ is capped by $L_n ( \pi_n^* )$, we also have the explicit identity
\begin{equation}\label{eq:Optimality-martingale-CAP}
\E [ L_n ( \pi_n^* ) \ | \ \F_{ i } ]= L_i (\pi_n^*) + v_{ n - i }( M_i ),
\end{equation}
which is often useful.

\subsection*{\sc Conditional Variances}\subsectionnewline

In \eqref{eq:Optimality-martingale}, the term $ v_{ n - i }( M_i ) = \E  [ L_n ( \pi_n^* ) - L_i( \pi_n^* ) \ | \ \F_{ i } ] $
tells us the expected number of selections that the policy $\pi_n^*$ will make from $\{ X_{i+1}, X_{i+2}, \ldots, X_n \}$
given the current value $M_i$ of the running maximum. There is a useful notion of \emph{conditional variance} that
is perfectly analogous. Specifically, we set
\begin{eqnarray}\label{eq:wf-identity}
w_{n-i} (M_{i}) \!\!\!\!  & \equalbydef & \!\!\!\! \Var[ L_n(\pi^*_n) - L_i (\pi^*_n) \,|\, \F_i] \\
                \!\!\!\!   & = & \!\!\!\!  \E[ \{L_n(\pi^*_n) - L_i (\pi^*_n) - v_{n-i} (M_i) \}^2 \,|\, \F_i ]. \notag
\end{eqnarray}
Here, of course, if $i = 0$ we always have $M_0 = 0$ and
$$
w_{n} (M_{0}) = \Var[ L_n(\pi^*_n)].
$$

The martingale $\{Y_i, \F_i\}_{i=0}^n$ defined by \eqref{eq:Optimality-martingale} leads in a natural way to
an informative representation
for the conditional variance \eqref{eq:wf-identity}, and one starts with
the difference sequence
\begin{equation}\label{eq:MDS-general-f}
d_j = Y_j - Y_{j-1}, \quad \quad \text{where } 1 \leq j \leq n.
\end{equation}
By \eqref{eq:Optimality-martingale} and telescoping of the sum we have
\begin{equation}\label{eq:partial-sum-MDS-general-f}
\sum_{j = i+1}^n d_j = L_n ( \pi^*_n ) - L_i ( \pi^*_n ) - v_{n-i} (M_i), \quad \quad \text{ for } 0 \leq i \leq n,
\end{equation}
so by orthogonality of the martingale differences we get
\begin{equation}\label{eq:variance-functions-general-f-BASIC}
w_{n-i} (M_{i}) = \Var[ L_n(\pi^*_n) - L_i (\pi^*_n) \,|\, \F_i] = \sum_{j = i+1}^n \E[ d_j^2 \,|\, \F_i ].
\end{equation}

This representation for the conditional variance $w_{n-i}$ can be usefully reframed
by taking a more structured view of the martingale differences \eqref{eq:MDS-general-f}.
Specifically, we write
\begin{equation}\label{eq:MDS-decomposition-A-B}
d_j = A_j + B_j,
\end{equation}
where the variable
\begin{equation}\label{eq:definition-Rj}
B_j \equalbydef v_{ n - j }( M_{j-1} ) - v_{ n - j + 1 }( M_{j-1} )
\end{equation}
represents the change in the martingale $ Y_j $
when we do not select $ X_j $, and where
\begin{equation}\label{eq:definition-Aj}
A_j \equalbydef ( 1 + v_{ n -j }( X_j ) - v_{ n - j}( M_{ j - 1 } ) ) \1 ( X_j \in [ M_{ j -1 }, h_{ n - j + 1 }( M_{ j - 1 } ) ] )
\end{equation}
is the \emph{additional} contribution to the change in the martingale $ Y_j $ when we do select $ X_j $.
Since $ B_j $ is $ \F_{ j - 1 } $-measurable, we have
$$
\E  [ d_j^2 \ | \ \F_{ j - 1 } ] = \E  [ A_j^2 \ | \ \F_{ j - 1 } ] + 2  B_j \ \E  [ A_j \ | \ \F_{ j - 1 } ] + B_j^2,
$$
and we also have $ 0 = \E  [ d_j \ | \ \F_{ j - 1 } ] = B_j + \E  [ A_j \ | \ \F_{ j - 1 } ] $, so
\begin{equation}\label{eq:martingale-difference-decomposition-general-f}
\E  [ d_j^2 \ | \ \F_{ j - 1 } ] = \E  [ A_j^2 \ | \ \F_{ j - 1 } ] - B_j^2.
\end{equation}

Now, for $j=i+1$ to $n$, we take the conditional expectation in \eqref{eq:martingale-difference-decomposition-general-f}
with respect to $\F_i$. When we sum these terms and recall \eqref{eq:variance-functions-general-f-BASIC}
we get our final representation for conditional variances
\begin{equation}\label{eq:variance-functions-general-f}
w_{n-i} (M_{i}) = \Var[ L_n(\pi^*_n) - L_i (\pi^*_n) \,|\, \F_i] = \sum_{j = i+1}^n \{ \E[ A_j^2 \,|\, \F_i ] - \E[B_j^2 \ | \ \F_i] \}.
\end{equation}

The decomposition \eqref{eq:variance-functions-general-f} was our main goal here,
but before concluding the section we should make one further inference from \eqref{eq:MDS-decomposition-A-B}.
By the defining representation \eqref{eq:optimal-threshold-general-f}
for $h_{k}$ we have $0 \leq A_j \leq 1$,
and by our bound \eqref{eq:simple-difference-bounds} on the value function differences we have
$ -1 \leq B_j \leq 0$.
Hence one has a uniform bound on the martingale differences
\begin{equation}\label{eq:MDS-bounded}
 | d_j | = | A_j + B_j | \leq 1 \quad \quad \text{ for } 1 \leq j \leq n.
\end{equation}

\section{\sc Inferences from the Uniform Model}\label{se:inferences-uniform-model}

We now consider the decompositions of Section \ref{se:probabilistic-interpretation-and-Optimality-martingale}
when $F$ is the uniform distribution on $[0,1]$, and we use superscripts to
make this specialization explicit. In particular, we let $\Xu_1, \Xu_2, \ldots, \Xu_n$ be the underlying sequence of $n$ independent
uniformly distributed random variables, and we write  $\Mu_{ i}$ for the value of the last observation
selected up to and including time $i\geq 1$ (and, as usual, we set $\Mu_0 = 0$).
Lemma \ref{lm:value-function-concave-uniform} tells us that the value function $\vu_k$ is concave,
and this is crucial to the proof of the lower bound for the conditional variance of
$\Lu_n(\pi^*_n)$.

\begin{proposition}[Conditional Variance Lower Bound] \label{pr:variance-lower-bound-uniform}
For $ 0 \leq i \leq n$ one has
\begin{equation*}
\frac{1}{3}  \, \vu_{n-i}( \Mu_i ) - 2 \leq \wu_{n-i} (\Mu_{i}).
\end{equation*}
\end{proposition}

\begin{proof}
Specialization of the representation \eqref{eq:variance-functions-general-f} gives us
\begin{equation}\label{eq:variance-mds-uniform}
\wu_{n-i} (\Mu_{i}) = \sum_{j = i+1}^n \E[ (\Au_j)^2 \,|\, \F_i ] -  \sum_{j = i+1}^n  \E[ (\Bu_j)^2 \ | \ \F_i],
\end{equation}
where  the definitions \eqref{eq:definition-Rj} and \eqref{eq:definition-Aj} now become
\begin{equation*}
\Bu_j = \vu_{ n - j }( \Mu_{j-1} ) - \vu_{ n - j + 1 }( \Mu_{j-1} )
\end{equation*}
and
\begin{equation}\label{eq:definition-Aj-uniform}
\Au_j = ( 1 + \vu_{ n -j }( X_j ) - \vu_{ n - j}( \Mu_{ j - 1 } ) ) \1 ( \Xu_j \in [ \Mu_{ j -1 }, \hu_{ n - j + 1 }( \Mu_{ j - 1 } ) ] ).
\end{equation}

First, we work toward a lower bound for the leading sum in \eqref{eq:variance-mds-uniform}.
If we square both sides of \eqref{eq:definition-Aj-uniform} and take conditional expectations, then we have
\begin{equation}\label{eq:Aj-integral-uniform}
\E  [ (\Au_j)^2 \ | \ \F_{ j - 1 } ]
= \int_{\Mu_{ j -1 }}^{\hu_{ n - j + 1 }( \Mu_{ j - 1 } )}  \{ 1 + \vu_{ n -j }( x ) - \vu_{ n - j}( \Mu_{ j - 1 } ) \}^2 \, dx.
\end{equation}
By Lemma \ref{lm:value-function-concave-uniform} the map $x \mapsto \vu_{ n -j }( x )$
is concave in $x$, so the line through the points $(\Mu_{j-1}, 1)$ and $(\hu_{n-j+1}( \Mu_{j-1}), 0)$
provides a lower bound on the integrand in \eqref{eq:Aj-integral-uniform}.
Integration of this linear lower bound then gives
\begin{equation}\label{eq:var-lb1}
\frac{1}{3}  \left(\hu_{n-j+1}( \Mu_{j-1}) - \Mu_{j-1} \right)
\leq \int_{\Mu_{ j -1 }}^{\hu_{ n - j + 1 }( \Mu_{ j - 1 } )}
\!\!\!\!\!\!\!\!\!\!\!\!\!\!\!\!\!\!\!\!\!\!
\{ 1 + \vu_{ n -j }( x ) - \vu_{ n - j}( \Mu_{ j - 1 } ) \}^2 \, dx.
\end{equation}
From the definition of $\Lu_i (\pi^*_n)$ we have the identity
$$
\sum_{j = i+1}^n \E[ \hu_{n-j+1}( \Mu_{j-1}) - \Mu_{j-1} \,|\, \F_i ] = \E[ \Lu_n(\pi^*_n) - \Lu_i (\pi^*_n) \,|\, \F_i] = \vu_{n-i} (\Mu_i),
$$
so \eqref{eq:Aj-integral-uniform} and \eqref{eq:var-lb1} give us
\begin{equation}\label{eq:Lower-Bound-on-As}
\frac{1}{3} \vu_{n-i} (\Mu_i) \leq \sum_{j = i+1}^n \E[ (\Au_j)^2 \,|\, \F_i ].
\end{equation}

Now, to work toward an upper bound on $\E[ (\Bu_j)^2 \ | \ \F_i]$,
we first note by the crude Lemma \ref{lm:Bellman-bounds} that
\begin{equation}\label{eq:Wj-squared-bound}
(\Bu_j)^2  = \left(  \vu_{ n - j }( \Mu_{j-1} ) - \vu_{ n - j + 1 }( \Mu_{j-1} ) \right)^2
      \leq \left(  \hu_{n-j+1}( \Mu_{j-1}) - \Mu_{j-1} \right)^2.
\end{equation}
The definition \eqref{eq:running-maximum-general-f} of the running maximum $\Mu_j$,
the uniform distribution of $\Xu_j$,  and calculus give us the identity
\begin{align}\label{eq:running-max-increment-uniform}
\E[\Mu_j -\Mu_{j-1}\, | \, \F_{j-1} ]
& = \int_{\Mu_{j-1}}^{\hu_{n-j+1}( \Mu_{j-1} )} (x -  \Mu_{j-1}) \, dx \\
& = \frac{1}{2} \left(\hu_{n-j+1}( \Mu_{j-1} )-\Mu_{j-1}\right)^2, \nonumber
\end{align}
so \eqref{eq:Wj-squared-bound} gives us the succinct bound
\begin{equation*}
(\Bu_j)^2 \leq 2 \, \E[\Mu_j -\Mu_{j-1}\, | \, \F_{j-1} ].
\end{equation*}
Now we take the conditional expectation with respect to $\F_{i}$ and sum over $i < j \leq n$.  Telescoping then gives us
\begin{equation}\label{eq:Upper-Bound-on-Bs}
\sum_{j = i+1}^n  \E[ (\Bu_j)^2 \ | \ \F_i] \leq 2 \, \E[\Mu_n -\Mu_{i}\, | \, \F_{i} ] \leq 2,
\end{equation}
where, in the last step, we used  $0 \leq \Mu_i \leq \Mu_n \leq 1$. The representation \eqref{eq:variance-mds-uniform}
and the bounds \eqref{eq:Lower-Bound-on-As} and \eqref{eq:Upper-Bound-on-Bs} complete the proof of the lemma.
\end{proof}

\subsection*{\sc A Cauchy-Schwarz Argument}\subsectionnewline

If we take the total expectation in \eqref{eq:running-max-increment-uniform}, then we have
\begin{equation}\label{eq:ProbsSquared}
\E[\left(\hu_{n-j+1}( \Mu_{j-1} )-\Mu_{j-1}\right)^2] = 2 \{ \E[\Mu_j] -\E[\Mu_{j-1}] \},
\end{equation}
and, since $\E[ (\hu_{n-j+1}( \Mu_{j-1} )-\Mu_{j-1} )]$ is the unconditional probability that
we accept the $j$'th element of the sequence, one might hope to estimate
$\E[L_n(\pi^*_n)]$ with help from \eqref{eq:ProbsSquared}  and a Cauchy-Schwarz argument.

In fact, by  two applications of the Cauchy-Schwarz inequality, we get
\begin{align*}
\E[ \Lu_n(\pi^*_n)]&=\sum_{j=1}^n \E[ \hu_{n-j+1}( \Mu_{j-1} )-\Mu_{j-1} ] \\
&\leq n^{1/2} \left\{ \sum_{j=1}^n (\E[\hu_{n-j+1}( \Mu_{j-1} )-\Mu_{j-1}])^2 \right\}^{1/2} \\
&\leq n^{1/2} \left\{ \sum_{j=1}^n \E\left[\left(\hu_{n-j+1}( \Mu_{j-1} )-\Mu_{j-1}\right)^2 \right] \right\}^{1/2},
\end{align*}
and, when we replace all of the summands using  \eqref{eq:ProbsSquared}, we get a telescoping sum
\begin{align*}
\E[ \Lu_n(\pi^*_n)]\leq n^{1/2} \left\{ 2 \sum_{j=1}^n  \{ \E[\Mu_j] -\E[\Mu_{j-1}] \} \right\}^{1/2}
=(2n)^{1/2} \{\E [\Mu_n] \}^{1/2}.
\end{align*}
We have $\E[ \Mu_n]<1$ since the support of $\Mu_n$ equals $[0,1]$, and, since the distribution of $\Lu_n(\pi^*_n)$ does not depend on $F$,
we find for all continuous $F$ that
\begin{equation}\label{eq:mean-bound-uniform-last}
\E[ L_n(\pi^*_n)]< (2n)^{1/2}.
\end{equation}
This recaptures the mean upper bound \eqref{eq:GnedinUpperBound} of \citeasnoun{BruRob:AAP1991} and \citeasnoun{Gne:JAP1999}
which was discussed in the introduction.

Here we should note that
\citeasnoun{BruDel:SPA2001} also used a Cauchy-Schwarz argument to show that for Poisson $N_\nu$ with mean $\nu$, one has
the analogous inequality
\begin{equation}\label{eq:BD-optimal-mean-Upper}
\E[L_{N_\nu} (\pi^*_{N_\nu})] \leq (2 \nu)^{1/2}.
\end{equation}
We know by
Proposition \ref{prop:info-lower-bound} that for $\nu=n$
we have $\E[L_{N_n} (\pi^*_{N_n})] \leq \E[ L_n(\pi^*_n)]$
but, even so, the bound \eqref{eq:BD-optimal-mean-Upper} does not help directly with \eqref{eq:mean-bound-uniform-last} ---
or vice versa. In addition to the usual issue that ``policies do not de-Poissonize," there is the real \emph{a priori} possibility that
$\E[L_{N_n}(\pi^*_{N_n})]$ might be much smaller than $\E[ L_n(\pi^*_n)]$.

\section{Inferences from the Exponential Model}\label{se:inferences-exponential-model}

Now we consider the exponential distribution $F(x)=1-e^{-x}$, for $x\geq 0$, and, as before,
we use superscripts to make this specialization explicit.
Thus $\Xe_1, \Xe_2, \ldots, \Xe_n$ denotes a sequence of $n$ independent, mean one, exponential
random variables, and $\Me_{ i}$ denotes the value of the last observation
selected up to and including time $i\geq 1$ (and, again, we set $\Me_0 = 0$).
This time Lemma \ref{lm:value-functions-exponential-convexity} provides the critical fact; it tells us that the value function $\ve_k$ is convex,
and this is at the heart of the argument.

\begin{proposition}[Conditional Variance Upper Bound]\label{pr:Variance-upper-bound-exponential}
For each $ 0 \leq i \leq n$ one has
\begin{equation}\label{eq:variance-function-upper-bound-exponential}
\we_{n-i} (\Me_{i}) \leq \frac{1}{3} \, \ve_{n-i}( \Me_i ) + \frac{2}{3} \{ 1 + \log (n - i) \}.
\end{equation}
\end{proposition}

The proof roughly parallels
that of Proposition \ref{pr:variance-lower-bound-uniform}, but in this case some
integrals are more troublesome to estimate.
To keep the argument direct, we extract one calculation as a lemma.

\begin{lemma}\label{lm:linear-approximation-bound}
For $0 \leq s < t < \infty$ one has
$$
\int_s^t \left(\frac{t-x}{t-s}\right)^2 e^{-x} \, dx
\leq \frac{ 1 }{ 3 } ( e^{-s} - e^{-t} ) + \frac{ 2 }{ 3 } \{ e^{-s} - e^{-t} (t - s + 1) \}.
$$
\end{lemma}

\begin{proof}
If we set
$$
g (y) \equalbydef  y^{-2 }  \{ - 6  y  + e^{ - y }( 2 y^3  + 3 y^2 - 6 + 6 e^{y} ) \} \quad \quad \text{for } y \geq 0,
$$
then by integration and simplification  one has for $0 \leq s < t < \infty$ that
$$
\frac{e^{-s}}{3} g ( t-s ) =
\int_s^t \left(\frac{t-x}{t-s}\right)^2 e^{-x} \, dx - \frac{ 1 }{ 3 } ( e^{-s} - e^{-t} )
- \frac{ 2 }{ 3 } \{ e^{-s} - e^{-t} (t - s + 1) \},
$$
and the lemma follows if we verify that $g (y) \leq 0$ for all $y \geq 0$.
By the integral representation
$$
g (y) = y^{-2} \left(  \int_0^y (- 6 ) \, dx + \int_0^y e^{-x}( 6 + 6x +3x^2 -2 x^3 )\, dx \right),
$$
we see that it suffices to show that
\begin{equation*}
6 + 6x +3x^2 -2 x^3 \leq 6 e^x \quad \quad \text{for all } x \in [0,\infty),
\end{equation*}
and the last inequality is obvious from the power series of $e^x$.
\end{proof}

\begin{proof}[Proof of Proposition \ref{pr:Variance-upper-bound-exponential}]
Specialization of  \eqref{eq:variance-functions-general-f} to the exponential model and simplification give us
\begin{equation}\label{eq:variance-mds-exponential}
\we_{n-i} (\Me_{i}) = \Var[ \Le_n(\pi^*_n) - \Le_i (\pi^*_n) \,|\, \F_i] \leq \sum_{j = i+1}^n \E[ (\Ae_j)^2 \,|\, \F_i ],
\end{equation}
where $\Ae_j$ is given by
\begin{equation*}
\Ae_j = ( 1 + \ve_{ n -j }( \Xe_j ) - \ve_{ n - j}( \Me_{ j - 1 } ) ) \1 ( \Xe_j \in [ \Me_{ j -1 }, \he_{ n - j + 1 }( \Me_{ j - 1 } ) ] ).
\end{equation*}
Since $ \Me_{ j - 1 } $ is $ \F_{ j - 1 }$-measurable, we have
\begin{equation}\label{eq:integration-Ai}
 \E  [ (\Ae_j)^2 \, | \ \F_{ j - 1 } ]
 = \int_{ \Me_{ j - 1 } }^{ \he_{ n - j + 1 }( \Me_{ j - 1} ) } \{ 1 + \ve_{ n - j }( x ) - \ve_{ n - j }( \Me_{ j - 1 } ) \}^2 e^{ -x } \ dx,
\end{equation}
and by Lemma \ref{lm:value-functions-exponential-convexity} the
map $x \mapsto  1 + \ve_{ n - j }( x ) - \ve_{ n - j }( \Me_{ j - 1 } )$
is convex in $x$ and non-negative for all $x \in [ \Me_{ j - 1} , \he_{ n - j + 1 }( \Me_{ j - 1 } ) ]$.

If $\he_{ n - j + 1 }( \Me_{ j - 1 } ) < \infty$, the line through the left-end point $(\Me_{ j - 1 },1)$
and the right-end point $(\he_{ n - j + 1 }( \Me_{ j - 1} ) ,0)$
provides us with an easy upper bound for the integrand \eqref{eq:integration-Ai}.
Specifically, for $ x \in [ \Me_{ j - 1} , \he_{ n - j + 1 }( \Me_{ j - 1 } ) ]$, we have that
\begin{equation}\label{eq:integration-Ai-piece2}
\{ 1 + \ve_{ n - j }( x ) - \ve_{ n - j }( \Me_{ j - 1 } ) \}^2 e^{ -x }
\le \left( \frac{ \he_{ n - j + 1 }( \Me_{ j - 1 } ) - x }{ \he_{ n - j + 1 }( \Me_{ j - 1 } ) - \Me_{ j - 1} } \right)^2 e^{ -x }.
\end{equation}
On the other hand, if $\he_{ n - j + 1 }( \Me_{ j - 1 } ) = \infty$, the right-side of \eqref{eq:integration-Ai-piece2}
is replaced by $e^{-x}$, and \eqref{eq:integration-Ai-piece2} again holds since
$0 \leq \{ 1 + \ve_{ n - j }( x ) - \ve_{ n - j }( \Me_{ j - 1 } ) \}\leq 1$.

We now integrate \eqref{eq:integration-Ai-piece2} and use the bound of Lemma \ref{lm:linear-approximation-bound};
the representation \eqref{eq:integration-Ai} then gives us
\begin{align}\label{eq:Zi-bound-explicit-exponential}
 \E  [ (\Ae_j)^2 \, | \ \F_{ j - 1 } ] & \leq  \frac{1}{3} ( e^{-\Me_{ j - 1 }} - e^{ -\he_{ n - j + 1 }( \Me_{ j - 1 } ) } )  \\
  & +  \frac{2}{3} \{ e^{-\Me_{ j - 1 }} - e^{ -\he_{ n - j + 1 }( \Me_{ j - 1 } ) } ( 1 + \he_{ n - j + 1 }( \Me_{ j - 1 } ) - \Me_{ j - 1 } )\}. \nonumber
\end{align}

Now we just need to interpret the two addends on the right-hand side of \eqref{eq:Zi-bound-explicit-exponential}.
The first addend is just the probability that observation $\Xe_j$
is selected when the value of the running maximum is $ \Me_{ j - 1 }$, that is,
\begin{equation}\label{eq:probability-accept-exponential}
\E [ \1 ( \Xe_j \in  [ \Me_{ j - 1 }, \he_{ n - j + 1 }( \Me_{ j - 1 } ) ] ) \, | \, \F_{j-1} ] =  e^{-\Me_{ j - 1 }} - e^{ -\he_{ n - j + 1 }( \Me_{ j - 1 } ) }.
\end{equation}
Similarly, the second addend of \eqref{eq:Zi-bound-explicit-exponential} is the one-period expected increment of the current
running maximum $\Me_{ j - 1 }$, or, to be explicit,
\begin{align}\label{eq:running-max-increment-exponential}
\E  [ \Me_j - \Me_{ j - 1 } & \ | \ \F_{ j - 1 } ] = \int_{ \Me_{ j - 1 } }^{ \he_{ n - j + 1 }( \Me_{ j - 1} ) } ( x - \Me_{ j - 1 } ) e^{ -x } \ dx \\
&= e^{-\Me_{ j - 1 }} - e^{ -\he_{ n - j + 1 }( \Me_{ j - 1 } ) }  ( 1 + \he_{ n - j + 1 }( \Me_{ j - 1 } ) - \Me_{ j - 1 } ). \nonumber
\end{align}
Given the two interpretations
\eqref{eq:probability-accept-exponential} and \eqref{eq:running-max-increment-exponential}, our
bound \eqref{eq:Zi-bound-explicit-exponential} now becomes
\begin{align*}
 \E  [ (\Ae_j)^2 \, | \ \F_{ j - 1 } ]
 \leq {} & \frac{1}{3} \ \E [ \1 ( \Xe_j \in  [ \Me_{ j - 1 }, \he_{ n - j + 1 }( \Me_{ j - 1 } ) ] ) \, | \, \F_{j-1} ] \\
         & + \frac{2}{3} \ \E  [ \Me_j - \Me_{ j - 1 }  \ | \ \F_{ j - 1 } ].
\end{align*}
Next, we recall the variance upper bound \eqref{eq:variance-mds-exponential},
take conditional expectations with respect to $\F_i$, and sum over $ i+1 \leq j \leq n$,
to obtain
\begin{align} \label{eq:variance-upper-bound-exponential-intermediate-step}
\we_{n-i} (\Me_{i})
& \leq \frac{1}{3} \ \E [ \Le_n ( \pi_n^* ) - \Le_i ( \pi_n^* ) \, | \,  \F_{ i } ]
   + \frac{2}{3} \ \E [ \Me_n - \Me_{ i }  \ | \ \F_{ i } ] \\
& = \frac{1}{3} \ \ve_{n-i} (\Me_{i}) + \frac{2}{3} \ \E [ \Me_n - \Me_{ i }  \ | \ \F_{ i } ], \nonumber
\end{align}
where in the last step we used the martingale identity \eqref{eq:Optimality-martingale-CAP}.

To complete the proof, we only need to estimate the conditional expectation
$\E [ \Me_n - \Me_{ i }  \ | \ \F_{ i } ]$. We first set
$
\MM^*_{[i+1,n]} = \max \{ \Xe_{i+1} , \Xe_{i+2} , \ldots , \Xe_n \},
$
and then we note that
$$
\Me_n  - \Me_{ i } \leq \max\{\MM^*_{[i+1,n]}, \Me_{i}\} -  \Me_{ i } \leq \MM^*_{[i+1,n]}.
$$
When we take the conditional expectations and use the independence of
$\F_i$ and $\{ \Xe_{i+1} , \Xe_{i+2} , \ldots , \Xe_n \}$, we get
\begin{equation}\label{eq:Mn-upperbound}
\E [ \Me_n - \Me_{ i }  \ | \ \F_{ i } ] \leq \E[ \MM^*_{[i+1,n]} \, | \, \F_i ] = \E[ \MM^*_{[1,n-i]} ].
\end{equation}
The logarithmic bound for the last term is well-known, but, for completeness, we just note
$\P ( \MM^*_{[1,n-i]}  \leq t ) = ( 1 - e^{ -t } )^{n-i}$,
so
$$
\E  [\MM^*_{[1,n-i]} ] =  \int_0^{\infty} 1 - ( 1 - e^{ -t } )^{ n - i } \ dt = \sum_{ j = 1 }^{n-i} j^{ -1 } \leq 1 + \log(n - i).
$$
This last estimate then combines with the upper bounds \eqref{eq:variance-upper-bound-exponential-intermediate-step}
and \eqref{eq:Mn-upperbound} to complete the proof of \eqref{eq:variance-function-upper-bound-exponential}.
\end{proof}

\section{Combined Inferences: Variance Bounds in General} \label{se:Variance-Bounds-in-General}

The variance bounds obtained under the uniform and exponential models are almost immediately applicable to general continuous $F$.
One only needs to make an appropriate translation.

\begin{proposition} For any continuous $F$ and for all $0 \leq i \leq n$ one has
the conditional variance bounds
\begin{equation}\label{eq:Conditional-in-General}
\frac{1}{3} \, v_{n-i}( M_i ) - 2 \leq w_{n-i} (M_{i}) \leq \frac{1}{3} \, v_{n-i}( M_i ) + \frac{2}{3} \{ 1 + \log (n - i) \}.
\end{equation}
In particular, for $i=0$ one has $M_0 = 0$ and
\begin{equation} \label{eq:variance-in-general}
 \frac{1}{3}  \, \E[ L_n(\pi^*_n) ] - 2 \leq  \Var[ L_n(\pi^*_n) ] \leq \frac{1}{3}  \, \E[ L_n(\pi^*_n) ] + \frac{2}{3}\{1 + \log n\}.
\end{equation}
\end{proposition}

\begin{proof}
If $X_1, X_2, \ldots, X_n$ is a sequence of independent random variables with the continuous distribution $F$, then the familiar transformations
$$
\Xu_i \equalbydef F(X_i) \quad \text{and} \quad
\Xe_i \equalbydef - \log\{ 1 - F(X_i)\}
$$
define sequences that have the uniform and exponential distribution, respectively.
These transformations give us a dictionary that we can use to translate results between our models;
specifically we have:
\begin{alignat*}{4}
  v_k(s)& = \vu_k ( F(s) )               &&\quad\quad \text{and} \quad\quad &        v_k(s) & =  \ve_k ( - \log\{ 1 - F(s)\} ), \\
  \Mu_i &= F(M_i)                        &&\quad\quad \text{and} \quad\quad &        \Me_i  & = - \log\{ 1 - F(M_i)\}, \\
  w_{n-i}(M_i) &=  \wu_{n-i} ( \Mu_i  )  &&\quad\quad \text{and} \quad\quad & w_{n-i}(M_i)  & =  \we_{n-i} ( \Me_i  ).
\end{alignat*}
Proposition \ref{pr:variance-lower-bound-uniform} tells us that
\begin{equation*}\label{eq:variance-function-lower-bound-uniform-2}
\frac{1}{3}  \, \vu_{n-i}( \Mu_i ) - 2 \leq \wu_{n-i} (\Mu_{i}),
\end{equation*}
so the first column of the dictionary gives us the first inequality of \eqref{eq:Conditional-in-General}. Similarly,
Proposition \ref{pr:Variance-upper-bound-exponential} tells us
\begin{equation*}
\we_{n-i} (\Me_{i}) \leq \frac{1}{3} \, \ve_{n-i}( \Me_i ) + \frac{2}{3} \{ 1 + \log (n - i) \},
\end{equation*}
and the second column of the dictionary gives us the second inequality of \eqref{eq:Conditional-in-General}.
\end{proof}

\section{The Central Limit Theorem} \label{se:CLT}

Our proof of the central limit theorem for $L_n(\pi^*_n)$
depends on the most basic version of the martingale central limit theorem.
\citeasnoun{Bro:AMS1971}, \citeasnoun{McLei:AP1974}, and \citeasnoun{HalHey:AP1980}
all give variations containing this one.

\begin{proposition}[Martingale Central Limit Theorem]\label{pr:martingale-clt}
For each $n \geq 1$, we consider a martingale difference sequence $\{ Z_{n,j}: 1 \leq j \leq n\}$ with respect
to the sequence of increasing $\sigma$-fields $\{\F_{n,j} : 0 \leq j \leq n\}$.
If
\begin{equation}\label{eq:negligibility}
\max_{1 \leq j \leq n} \parallel Z_{n,j} \parallel_\infty \rightarrow 0 \quad \quad \text{ as } n \rightarrow \infty
\end{equation}
and
\begin{equation}\label{eq:LLN-conditional-variance}
\sum_{j=1}^n \E[ Z^2_{n,j} | \F_{j-1}] \stackrel{\rm p}{\longrightarrow} 1 \quad \quad \text{ as } n \rightarrow \infty,
\end{equation}
then  we have the convergence in distribution
$$
\sum_{j=1}^n Z_{n,j} \Longrightarrow N(0,1) \quad \quad \text{ as }n \rightarrow \infty.
$$
\end{proposition}

For each $n \geq 1$, we consider a driving sequence
$X_{n,1}, X_{n,2}, \ldots, X_{n,n}$
of independent random variables with the continuous
distribution $F$. We then set
\begin{equation*}
Z_{n,j} \equalbydef \frac{3^{1/2} d_{n,j}}{(2n)^{1/4}}, \quad \quad \text{for } 1 \leq j \leq n,
\end{equation*}
where the $d_{n,j}$'s are the differences defined by \eqref{eq:MDS-general-f}, although we now  make explicit
the dependence of the differences on $n$. This is a  martingale
difference sequence with respect to the increasing sequence of $\sigma$-fields $\F_{n,j} = \sigma\{ X_{n,1}, X_{n,2}, \ldots, X_{n,j}\}$,
and when we take $i = 0$ in \eqref{eq:partial-sum-MDS-general-f} we get the basic representation
$$
\sum_{j=1}^n Z_{n,j} = \frac{3^{1/2} \{ L_n(\pi^*_n) - \E[L_n(\pi^*_n)] \} }{(2n)^{1/4}}.
$$

We know from \eqref{eq:MDS-bounded} that we always have
$| d_{n,j} | \leq 1 $ so, by our normalization,
the negligibility condition \eqref{eq:negligibility} is trivially valid.
The heart of the matter is the proof of the weak law \eqref{eq:LLN-conditional-variance};
more explicitly, we need to show that
\begin{equation}\label{eq:weak-law-conditional-variances}
\sum_{j=1}^n \frac{3 \, \E[ d_{n,j}^2  \ | \ \F_{n,j-1} ]  }{(2n)^{1/2}}
\stackrel{\rm p}{\longrightarrow} 1 \quad \quad \text{ as } n \rightarrow \infty.
\end{equation}
The variance bounds \eqref{eq:variance-in-general} and the asymptotic relation \eqref{eq:SSasymp} for the mean
imply
$$
\Var[L_n(\pi^*_n)] \sim \frac{1}{3} \E[L_n(\pi^*_n)] \sim \frac{(2n)^{1/2}}{3}  \quad \quad \text{ as } n \rightarrow \infty,
$$
and telescoping and orthogonality of the differences $\{ d_{n,j}: 1 \leq j \leq n\}$ give us
$$
\Var[L_n(\pi^*_n)] =\E \bigg[ \sum_{j=1}^n \E[ d_{n,j}^2] \bigg] = \E \bigg[ \sum_{j=1}^n \E[ d_{n,j}^2  \ | \ \F_{n,j-1} ] \bigg],
$$
so the weak law \eqref{eq:weak-law-conditional-variances}
will follow from Chebyshev's inequality if one proves that
\begin{equation}\label{eq:Little-Oh-Last}
\Var \bigg[ \sum_{j=1}^n \E[ d_{n,j}^2  \ | \ \F_{n,j-1} ] \bigg] = o(n) \quad \quad \text{ as } n \rightarrow \infty.
\end{equation}

The proof of Theorem 1 is completed once one confirms the relation \eqref{eq:Little-Oh-Last}, and
the next lemma gives us more than we need.

\begin{lemma}[Conditional Variance Bound]
If $F$ is continuous, then for $n \geq 1$, one has
$$
\Var \bigg[ \sum_{j=1}^n \E[ d_{n,j}^2  \ | \ \F_{n,j-1} ] \bigg]  \leq \{18 + (\log n)^2 \} (2 n )^{1/2}.
$$
\end{lemma}

\begin{proof}
We fix $n \geq 1$ and simplify the notation by dropping the subscript $n$ on the martingale
difference sequence and the filtration. We then let
$$
V \equalbydef \sum_{j=1}^n \E[ d_j^2  \ | \ \F_{j-1} ]
$$
and consider the martingale $\{ V_i : 0 \leq i \leq n\}$
defined by setting
$$
 V_i \equalbydef \E [ V   \ | \ \F_i \ ] \quad \quad \text{for } 0 \leq i \leq n.
$$
One has the initial and terminal values
$$
V_0 = \sum_{j=1}^n \E[ d_j^2 ] = \Var[L_n(\pi^*_n)] \quad \quad \text{ and } \quad \quad V_n = V,
$$
and if we introduce the new martingale differences
$\Delta_i = V_i - V_{i-1},$
$1 \leq i \leq n$,
then telescoping and orthogonality give us
\begin{equation}\label{eq:variance-CLT-proof}
V_n - V_0 = \sum_{i=1}^n \Delta_i \quad \quad \text{and}
\quad \quad
\Var[V_n] = \Var \bigg[ \sum_{j=1}^n \E[ d_j^2  \ | \ \F_{j-1} ] \bigg]  = \sum_{i=1}^n \E[\Delta_i^2].
\end{equation}
For $1 \leq j \leq i+1$ all of the summands $\E[ d_j^2  \ | \ \F_{j-1} ]$ are  $\F_i$-measurable,
so we have
\begin{align*}
\Delta_i  = & \sum_{j=1}^i \E[ d_j^2  \ | \ \F_{j-1} ] + \E\bigg[ \sum_{j=i+1}^n \E[ d_j^2  \ | \ \F_{j-1} ] \ | \ \F_i \bigg] \\
            & - \sum_{j=1}^i \E[ d_j^2  \ | \ \F_{j-1} ] - \E\bigg[ \sum_{j=i+1}^n \E[ d_j^2  \ | \ \F_{j-1} ] \ | \ \F_{i-1} \bigg].
\end{align*}
The first and the third sum cancel, and we obtain
\begin{align} \label{eq:Delta_i-identity}
\Delta_i & = \sum_{j=i+1}^n \E[ d_j^2  \ | \ \F_i ] - \E \bigg[  \sum_{j=i+1}^n \E[ d_j^2  \ | \ \F_{i} ]  \ | \ \F_{i-1} \bigg]\\
         & = w_{n-i}(M_i) - \E[w_{n-i}(M_i) \ | \  \F_{i-1}], \nonumber
\end{align}
where in the last line we twice used the formula \eqref{eq:variance-functions-general-f-BASIC} for the conditional variance.

Next, we set
$$G_i \equalbydef \{ \omega: X_i(\omega) \in [M_{i-1}(\omega), h_{n-i+1}(M_{i-1}(\omega)) ] \},$$
so, in words,  $G_i$ is the set of all
$\omega$ for which the observation $X_i(\omega)$
is selected at time $i$ under the optimal policy $\pi_n^*$.
By the recursive definition \eqref{eq:running-maximum-general-f}
of the running maximum $M_i$, we then have the decomposition
\begin{equation}\label{eq:conditional-variance-decomposition}
w_{n-i}(M_i) = w_{n-i}(M_{i-1}) + \{ w_{n-i}(X_i) - w_{n-i}(M_{i-1}) \}  \1( G_i ).
\end{equation}
In fact, one can replace $w_{n-i}$ with any function here, and it will be useful to also note that
\begin{equation}\label{eq:value-function-decomposition}
v_{n-i}(M_i) = v_{n-i}(M_{i-1}) + \{ v_{n-i}(X_i) - v_{n-i}(M_{i-1}) \}  \1( G_i ).
\end{equation}

The first summand on the right-hand side of \eqref{eq:conditional-variance-decomposition}
is $\F_{i-1}$-measurable, so, if we rewrite \eqref{eq:Delta_i-identity}
using \eqref{eq:conditional-variance-decomposition} we obtain
$$
\Delta_i = \{ w_{n-i}(X_i) - w_{n-i}(M_{i-1}) \}  \1( G_i ) - \E[ \{ w_{n-i}(X_i) - w_{n-i}(M_{i-1}) \}  \1( G_i ) \ | \ \F_{i-1}].
$$
When we square this identity and take the conditional expectation we find
\begin{equation}\label{eq:Delta-i-second-moment-upper-bound}
\E[\Delta_i^2 \ | \ \F_{i-1}]   \leq \E[ \{ w_{n-i}(X_i) - w_{n-i}(M_{i-1}) \}^2  \1( G_i ) \ | \ \F_{i-1}],
\end{equation}
and all that remains is to estimate the difference $\{ w_{n-i}(X_i) - w_{n-i}(M_{i-1}) \}  \1( G_i )$.

Now consider the \emph{upper bound} in \eqref{eq:Conditional-in-General}
and replace $w_{n-i}(M_i)$ and $v_{n-i}(M_i)$ with their decompositions
\eqref{eq:conditional-variance-decomposition} and \eqref{eq:value-function-decomposition}. When we move the
term $w_{n-i}(M_{i-1})$ to the right side, we have
\begin{align}\label{eq:clt-proof-bound0}
 \{ w_{n-i}(X_i) - w_{n-i}(M_{i-1}) \}  \1( G_i )
  \leq {} & \frac{1}{3}\{ v_{n-i}(X_i) - v_{n-i}(M_{i-1}) \}  \1( G_i )\\
 &  + \frac{1}{3}v_{n-i}(M_{i-1}) - w_{n-i}(M_{i-1}) \nonumber\\
 &  + \frac{2}{3} ( 1 + \log n ). \nonumber
\end{align}
By the \emph{lower bound} in \eqref{eq:Conditional-in-General} the second summand is bounded by two, and,
to estimate the first summand, we note that  the characterization \eqref{eq:optimal-threshold-general-f}
for the optimal threshold function and the monotonicity of the value function give us
\begin{equation}\label{eq:clt-proof-bound1}
\frac{1}{3} | \, v_{n-i}(X_i) - v_{n-i}(M_{i-1}) \, |  \1( G_i ) \leq \frac{1}{3} \1( G_i ).
\end{equation}
The left-hand side of \eqref{eq:clt-proof-bound0} is zero off of the set $G_i$, so using \eqref{eq:clt-proof-bound1}
we see that \eqref{eq:clt-proof-bound0} gives us
\begin{equation}\label{eq:upper-bound-difference-cond-variance}
 \{ w_{n-i}(X_i) - w_{n-i}(M_{i-1}) \}  \1( G_i )  \leq \{ 3 + \frac{2}{3} \log n \} \1( G_i ).
\end{equation}
A parallel argument gives us the complementary inequality,
\begin{equation}\label{eq:lower-bound-difference-cond-variance}
- \{ 3 + \frac{2}{3} \log n \} \1( G_i ) \leq  \{ w_{n-i}(X_i) - w_{n-i}(M_{i-1}) \}  \1( G_i ).
\end{equation}
Specifically, one now begins with the
\emph{lower bound} in \eqref{eq:Conditional-in-General} and replaces
$w_{n-i}(M_i)$ and $v_{n-i}(M_i)$ with their decompositions.

Taken together \eqref{eq:upper-bound-difference-cond-variance} and \eqref{eq:lower-bound-difference-cond-variance}
imply
$$
| \, w_{n-i}(X_i) - w_{n-i}(M_{i-1}) \ | \   \1( G_i )  \leq \{ 3 + \frac{2}{3} \log n \} \1( G_i ),
$$
so we can square both sides and take conditional expectations with respect to $\F_{i-1}$.
By \eqref{eq:Delta-i-second-moment-upper-bound} and the definition of $G_i$, we then have
$$
\E[\Delta_i^2 \ | \ \F_{i-1}]  \leq \{ 18 + \frac{8}{9}( \log n )^2 \} \E[  \1( X_i \in [M_{i-1}, h_{n-i+1}(M_{i-1}) ] ) \ | \ \F_{i-1} ],
$$
so if we drop the factor $8/9$, take total expectations, and sum we get
\begin{equation}\label{eq:VarianceBoundLastLine}
\sum_{i=1}^n \E[\Delta_i^2] \leq \{ 18 + ( \log n )^2 \} \sum_{i=1}^n \E[ \1( X_i \in [M_{i-1}, h_{n-i+1}(M_{i-1}) ] )].
\end{equation}
By \eqref{eq:variance-CLT-proof} the sum on the left is variance of $V_n = \sum_{j=1}^n \E[ d_j^2  \ | \ \F_{j-1} ]$,
and by \eqref{eq:Ln} the sum of the expected values on the right is equal to $\E[L_n(\pi^*_n)]$.
Finally, we know
that $ \E[L_n(\pi^*_n)] < (2n)^{1/2}$ from \eqref{eq:GnedinUpperBound} and the argument of Section \ref{se:inferences-uniform-model},
so \eqref{eq:VarianceBoundLastLine} completes the proof of the lemma.
\end{proof}

\section{Concluding Observations}\label{se:Conclusions}

The idea of ``spending symmetry" that was mentioned in Section \ref{se:DP-special-properties}
originates with an instructive essay of \citeasnoun[Section 1.4]{Tao:AMS2009}. This notion
can be cast in stunning generality, but here it turns out to be resolutely concrete and very useful.

The variance lower bound of \eqref{eq:variance-in-general} had been known to us for some years, but dogged
analysis of the uniform model left us without an upper bound of comparable quality. A general Markov decision problem (MDP) bound
in \citeasnoun{ArlGanSte:OR2014} would give $\Var[L_n(\pi_n^*)] \leq \E[L_n(\pi_n^*)]$,
but here the MDP bound is too weak by a factor of three.
It cannot serve even as good motivation for a central limit theorem.

With such a long tradition of immediate reduction to the uniform model, it was surprising to see how fruitful it could be to
simultaneously use the exponential model --- even though the distribution of $L_n(\pi_n^*)$ is the \emph{same} under either model.
Still, with different value functions come different qualitative features, and the convexity of the value functions under the
exponential model leads in a natural way to the needed upper bound of the variance. This opened up the
way to the rest of the analysis.

We mentioned one open problem earlier (see Remark \ref{rem:concavity-n-offline}),
and there is a related
problem that deserves some thought. In the offline selection problem, the distribution of the length of the longest increasing subsequence of a sequence of $n$ independent
uniformly distributed random variables is the same as the distribution of the
length of the longest increasing subsequence of a random permutation
of the integers $\{1,2,\ldots, n\}$. This equivalence is lost in the online selection problem, and it is unclear how much of
Theorem \ref{thm:OurCLT} can be recaptured.

For example, if we write $L^{\rm perm}_n(\pi^*_n)$
for the analog of $L_n(\pi^*_n)$ where now one chooses a random permutation of $\{1,2,\ldots, n\}$,
then, by an argument of Burgess Davis given in \citeasnoun{SamSte:AP1981},
one does have $\E [L^{\rm perm}_n(\pi^*_n)] \sim (2n)^{1/2}$ as $n \rightarrow \infty$.
Unfortunately, mean bounds like those
of Theorem \ref{thm:OurCLT} cannot be achieved in this way, and variance bounds that would be good enough to support a central
limit theorem are even more remote. Nevertheless, some analog of Theorem \ref{thm:OurCLT} is quite likely to be true.


\end{document}